\documentclass{amsart}
\usepackage{amscd}
\numberwithin{equation}{section}
\usepackage{amssymb,latexsym}
\usepackage[utf8]{inputenc}

\begin{document}
% One author
%\title[shorttitle]{titleline1\\
%                   titleline2}
%\author{Gunnar Fl{\o}ystad}
%\address{Matematisk Institutt\\
%         Johs. Brunsgt. 12\\
%        5008 Bergen}

%\email{gunnar@mi.uib.no}
%\urladdr{http://webaddress}
%\thanks{thanks}
% End one author

% Two authors
\title
[Monomial algebras
defined by Lyndon words] {Monomial algebras
defined by Lyndon words}

\author{Tatiana Gateva-Ivanova}
\address{Institute of Mathematics and Informatics\\
         Bulgarian Academy of Sciences\\
         Sofia 1113\\
         Bulgaria}
\email{tatianagateva@yahoo.com; tatyana@aubg.bg}
%\urladdr{http://firstauthorwebaddress}
\thanks{The first author was partially supported by the Max Planck Institute for Mathematics, Bonn}
\author{Gunnar Fl{\o}ystad}
\address{Matematisk Institutt\\
         Johs. Brunsgt. 12\\
         5008 Bergen\\
         Norway}
\email{gunnar@mi.uib.no}
%\urladdr{http://secondauthorwebaddress}
%\thanks{The second author ... thanks}
% End two authors

\theoremstyle{plain}
\newtheorem{theorem}{Theorem}[section]
\newtheorem{corollary}[theorem]{Corollary}
\newtheorem*{main}{Main Theorem}
\newtheorem{lemma}[theorem]{Lemma}
\newtheorem{proposition}[theorem]{Proposition}
\newtheorem{conjecture}[theorem]{Conjecture}
\newtheorem{fact}[theorem]{Fact}
\newtheorem*{pieri}{Pieri's rule}
\newtheorem*{theoremA}{Theorem A}
\newtheorem*{theoremB}{Theorem B}
\newtheorem*{theoremC}{Theorem C}
\newtheorem*{proofA}{Proof of Theorem A}
\newtheorem*{proofB}{Proof of Theorem B}
\newtheorem*{proofC}{Proof of Theorem C}
\newtheorem*{problem}{Problem}
\newtheorem*{method}{Method}

\theoremstyle{definition}
\newtheorem{definition}[theorem]{Definition}

\theoremstyle{remark}
\newtheorem{notation}[theorem]{Notation}
\newtheorem{remark}[theorem]{Remark}
\newtheorem{example}[theorem]{Example}
\newtheorem{openquestion}[theorem]{Open Question}
\newtheorem{claim}{Claim}

%Mine egne.

\newcommand{\psp}[1]{{{\bf P}^{#1}}}
\newcommand{\psr}[1]{{\bf P}(#1)}
\newcommand{\opw}{\op_{\psr{W}}}
\newcommand{\go}{\op}

%Initial ideals
\newcommand{\ini}[1]{\text{in}(#1)}
\newcommand{\gin}[1]{\text{gin}(#1)}
\newcommand{\kr}{{\Bbbk}}
\newcommand{\kk}{{\Bbbk}}
\newcommand{\pd}{\partial}
\newcommand{\vardel}{\partial}
\renewcommand{\tt}{{\bf t}}

%Kategorier

\newcommand{\coh}{{{\text{{\rm coh}}}}}
\def\diag{\operatorname {diag}}
%Modulkategorier

\newcommand{\modv}[1]{{#1}\text{-{mod}}}
\newcommand{\modstab}[1]{{#1}-\underline{\text{mod}}}

\newcommand{\sut}{{}^{\tau}}
\newcommand{\sumit}{{}^{-\tau}}
\newcommand{\til}{\thicksim}

\newcommand{\totp}{\text{Tot}^{\prod}}
\newcommand{\dsum}{\bigoplus}
\newcommand{\dprod}{\prod}
\newcommand{\lsum}{\oplus}
\newcommand{\lprod}{\Pi}

% Algebraer
\newcommand{\La}{{\Lambda}}
\newcommand{\lam}{{\lambda}}
\newcommand{\GL}{{GL}}
\newcommand{\asX}{\langle X \rangle}

\newcommand{\sirstj}{\circledast}

% Knipper
\newcommand{\she}{\EuScript{S}\text{h}}
\newcommand{\cm}{\EuScript{CM}}
\newcommand{\cmd}{\EuScript{CM}^\dagger}
\newcommand{\cmri}{\EuScript{CM}^\circ}
\newcommand{\cler}{\EuScript{CL}}
\newcommand{\clerd}{\EuScript{CL}^\dagger}
\newcommand{\clerri}{\EuScript{CL}^\circ}
\newcommand{\gor}{\EuScript{G}}
\newcommand{\gA}{\mathcal{A}}
\newcommand{\gF}{\mathcal{F}}
\newcommand{\gG}{\mathcal{G}}
\newcommand{\gM}{\mathcal{M}}
\newcommand{\gE}{\mathcal{E}}
\newcommand{\gD}{\mathcal{D}}
\newcommand{\gI}{\mathcal{I}}
\newcommand{\gP}{\mathcal{P}}
\newcommand{\gK}{\mathcal{K}}
\newcommand{\gL}{\mathcal{L}}
\newcommand{\gS}{\mathcal{S}}
\newcommand{\gC}{\mathcal{C}}
\newcommand{\gO}{\mathcal{O}}
\newcommand{\gJ}{\mathcal{J}}
\newcommand{\gU}{\mathcal{U}}
\newcommand{\mm}{\mathfrak{m}}

\newcommand{\dlim} {\varinjlim}
\newcommand{\ilim} {\varprojlim}

%Kategorier
\newcommand{\CM}{\text{CM}}
\newcommand{\Mon}{\text{Mon}}

%Kategorieer av komplekser

\newcommand{\Kom}{\text{Kom}}

% Begreper homologisk alebra

\newcommand{\EH}{{\mathbf H}}
\newcommand{\res}{\text{res}}
\newcommand{\Hom}{\text{Hom}}
\newcommand{\inhom}{{\underline{\text{Hom}}}}
\newcommand{\Ext}{\text{Ext}}
\newcommand{\Tor}{\text{Tor}}
\newcommand{\ghom}{\mathcal{H}om}
\newcommand{\gext}{\mathcal{E}xt}
\newcommand{\id}{\text{{id}}}
\newcommand{\im}{\text{im}\,}
\newcommand{\codim} {\text{codim}\,}
\newcommand{\resol}{\text{resol}\,}
\newcommand{\rank}{\text{rank}\,}
\newcommand{\lpd}{\text{lpd}\,}
\newcommand{\coker}{\text{coker}\,}
\newcommand{\supp}{\text{supp}\,}
\newcommand{\Ad}{A_\cdot}
\newcommand{\Bd}{B_\cdot}
\newcommand{\Gd}{G_\cdot}

%Avbildninger og andre symbolforkortelser

\newcommand{\sus}{\subseteq}
\newcommand{\sups}{\supseteq}
\newcommand{\pil}{\rightarrow}
\newcommand{\vpil}{\leftarrow}
\newcommand{\lvpil}{\longleftarrow}
\newcommand{\rpil}{\leftarrow}
\newcommand{\lpil}{\longrightarrow}
\newcommand{\inpil}{\hookrightarrow}
\newcommand{\pils}{\twoheadrightarrow}
\newcommand{\projpil}{\dashrightarrow}
\newcommand{\dotpil}{\dashrightarrow}
\newcommand{\adj}[2]{\overset{#1}{\underset{#2}{\rightleftarrows}}}
\newcommand{\mto}[1]{\stackrel{#1}\longrightarrow}
\newcommand{\vmto}[1]{\overset{\tiny{#1}}{\longleftarrow}}
\newcommand{\mtoelm}[1]{\stackrel{#1}\mapsto}

\newcommand{\eqv}{\Leftrightarrow}
\newcommand{\impl}{\Rightarrow}

\newcommand{\iso}{\cong}
\newcommand{\te}{\otimes}
\newcommand{\tek}{\te_\kr}
\newcommand{\sqte}{\te}
\newcommand{\into}[1]{\hookrightarrow{#1}}
\newcommand{\ekv}{\Leftrightarrow}
\newcommand{\equi}{\simeq}
\newcommand{\isopil}{\overset{\cong}{\lpil}}
\newcommand{\equipil}{\overset{\equi}{\lpil}}
\newcommand{\ispil}{\isopil}
\newcommand{\vvi}{\langle}
\newcommand{\hvi}{\rangle}
\newcommand{\susneq}{\subsetneq}
\newcommand{\sgn}{\text{sign}}

%Notasjonsforkortelser

\newcommand{\xd}{\check{x}}
\newcommand{\ortog}{\bot}
\newcommand{\tL}{\tilde{L}}
\newcommand{\tM}{\tilde{M}}
\newcommand{\tH}{\tilde{H}}
\newcommand{\tvH}{\widetilde{H}}
\newcommand{\tvh}{\widetilde{h}}
\newcommand{\tV}{\tilde{V}}
\newcommand{\tS}{\tilde{S}}
\newcommand{\tT}{\tilde{T}}
\newcommand{\tR}{\tilde{R}}
\newcommand{\tf}{\tilde{f}}
\newcommand{\ts}{\tilde{s}}
\newcommand{\tp}{\tilde{p}}
\newcommand{\tr}{\tilde{r}}
\newcommand{\tfst}{\tilde{f}_*}
\newcommand{\empt}{\emptyset}
\newcommand{\bfa}{{\bf a}}
\newcommand{\bfb}{{\bf b}}
\newcommand{\bfd}{{\bf d}}
\newcommand{\bfe}{{\bf e}}
\newcommand{\bfp}{{\bf p}}
\newcommand{\bfc}{{\bf c}}
\newcommand{\bfl}{{\bf \ell}}
\newcommand{\bfz}{{\bf z}}
\newcommand{\bfj}{{\bf j}}
\newcommand{\ubfd}{\underline{\bfd}}
\newcommand{\la}{\lambda}
\newcommand{\bfen}{{\mathbf 1}}
\newcommand{\ep}{\epsilon}
\newcommand{\en}{r}
\newcommand{\tu}{s}
\newcommand{\integ}{{int}}
\newcommand{\module}{{mod}}
\newcommand{\inc}{{deg}}
\newcommand{\dec}{{root}}

\newcommand{\ome}{\omega_E}

\newcommand{\bevis}{{\bf Proof. }}
\newcommand{\demofin}{\qed \vskip 3.5mm}
\newcommand{\nyp}[1]{\noindent {\bf (#1)}}
\newcommand{\demo}{{\it Proof. }}
\newcommand{\demodone}{\demofin}
\newcommand{\parg}{{\vskip 2mm \addtocounter{theorem}{1}
                   \noindent {\bf \thetheorem .} \hskip 1.5mm }}

\newcommand{\lcm}{{\text{lcm}}}

% Simplisielle komplekser

\newcommand{\dl}{\Delta}
\newcommand{\cdel}{{C\Delta}}
\newcommand{\cdelp}{{C\Delta^{\prime}}}
\newcommand{\dlst}{\Delta^*}
\newcommand{\Sdl}{{\mathcal S}_{\dl}}
\newcommand{\lk}{\text{lk}}
\newcommand{\lkd}{\lk_\Delta}
\newcommand{\lkp}[2]{\lk_{#1} {#2}}
\newcommand{\del}{\Delta}
\newcommand{\delr}{\Delta_{-R}}
\newcommand{\dd}{{\dim \del}}

%Monomialidealer
\renewcommand{\aa}{{\bf a}}
\newcommand{\bb}{{\bf b}}
\newcommand{\cc}{{\bf c}}
\newcommand{\xx}{{\bf x}}
\newcommand{\yy}{{\bf y}}
\newcommand{\zz}{{\bf z}}
\newcommand{\mv}{{\xx^{\aa_v}}}
\newcommand{\mF}{{\xx^{\aa_F}}}

\newcommand{\proj}[1]{{\mathbb P}^{#1}}
\newcommand{\hele}{{\mathbb Z}}
\newcommand{\nat}{{\mathbb N}}
\newcommand{\rat}{{\mathbb Q}}

\newcommand{\pnm}{{\bf P}^{n-1}}
\newcommand{\opnm}{{\go_{\pnm}}}
\newcommand{\op}[1]{\gO_{\proj{#1}}}
\newcommand{\ompn}{\Omega_{\proj{n}}}
\newcommand{\opn}{\op{n}}
\newcommand{\opm}{\op{m}}
\newcommand{\ompnm}{\omega_{\pnm}}
\newcommand{\Ncal}{\mbox{$\cal N$}}

\newcommand{\dt}{{\displaystyle \cdot}}
\newcommand{\st}{\hskip 0.5mm {}^{\rule{0.4pt}{1.5mm}}}
\newcommand{\disk}{\scriptscriptstyle{\bullet}}

\newcommand{\cF}{F_\dt}
\newcommand{\Fd}{F_{\disk}}
\newcommand{\pol}{f}

\newcommand{\disc}{\circle*{5}}

%Lineære rom
\newcommand{\Dab}{{\mathbb D}(\bfa, \bfb)}
\newcommand{\Ddab}{B(\bfa, \bfb)}
\newcommand{\Tab}{{\mathbb T}(\bfa, \bfb)}
\newcommand{\Cab}{C(\bfa, \bfb)}
\newcommand{\Pab}{B(\bfa, \bfb)}
\newcommand{\Lhkab}{L^{HK}(\bfa, \bfb)}
\newcommand{\Lab}{L(\bfa, \bfb)}
\newcommand{\Sigab}{\Sigma(\bfa, \bfb)}
\newcommand{\hlow}{h_{low}}
\newcommand{\hup}{h_{up}}
\newcommand{\up}[2]{\hup}
\newcommand{\facet}[2]{{\bf facet}(#1, #2)}
\newcommand{\BS}{Boij-S\"oderberg}
\newcommand{\bstar}{{\mathbb B}^*}
\newcommand{\Symm}{\mbox{Symm}}
\newcommand{\zinc}[1]{{\mathbb Z}^{#1}_\inc}
\newcommand{\zdec}[1]{{\mathbb Z}^{#1}_\dec}
\newcommand{\fib}{f}
\newcommand{\asV}{\langle V \rangle}

\def\CC{{\mathbb C}}
\def\GG{{\mathbb G}}
\def\ZZ{{\mathbb Z}}
\def\NN{{\mathbb N}}
\def\RR{{\mathbb R}}
\def\OO{{\mathbb O}}
\def\QQ{{\mathbb Q}}
\def\VV{{\mathbb V}}
\def\PP{{\mathbb P}}
\def\EE{{\mathbb E}}
\def\FF{{\mathbb F}}
\def\AA{{\mathbb A}}

\keywords{Lyndon words, monomial algebras,  polynomial growth,
global dimension, Artin-Schelter regular algebras}
\subjclass[2010]
{Primary: 16P90, 16S15, 18G20, 68R15; Secondary:
 16S80, 16Z05}
\date{\today}

\begin{abstract}
Assume that  $X=  \{x_1,\cdots,x_g\}$ is a finite alphabet and $K$
is a field. We study monomial algebras $A= K \langle X \rangle
/(W)$, where $W$ is an antichain of Lyndon words in $X$ of
arbitrary cardinality. We find a Poincar\'{e}-Birkhoff-Witt type
basis of $A$ in terms of its \emph{Lyndon atoms} $N$, but, in
general, $N$ may be infinite. We prove that if  $A$ has polynomial
growth of degree $d$ then $A$ has global dimension $d$ and is
standard finitely presented, with $d-1 \leq |W| \leq d(d-1)/2$. Furthermore,
$A$ has polynomial growth \emph{iff} the set of Lyndon atoms $N$ is
finite. In this case  $A$ has a $K$-basis $\mathfrak{N} =
\{l_1^{\alpha_{1}}l_2^{\alpha_{2}}\cdots l_d^{\alpha_{d}} \mid
\alpha_{i} \geq 0,  1 \leq i \leq d\}$, where $N = \{l_1,
\cdots,l_d\}$. We give an extremal class of monomial algebras, the
Fibonacci-Lyndon algebras, $F_n$, with global dimension $n$ and
polynomial growth, and show that the algebra $F_6$ of global
dimension $6$ cannot be deformed, keeping the multigrading, to an
Artin-Schelter regular algebra.
\end{abstract}
\maketitle

\section{Introduction}

Let $X= \{x_1, x_2, \cdots,  x_g\}$  be a finite alphabet. Denote by
$X^*$ the free monoid generated by $X$, the empty word is denoted by
$1$.  $X^{+}$ is the free semigroup generated by $X$, $X^{+}= X^*
-\{1\}$. Throughout the paper  $K\asX$ stands for the free
associative $K$-algebra generated by $X$, where $K$ is a field. As
usual, the length of a word $w\in X^{+}$ is denoted by $|w|$. We
shall consider the canonical grading  on $K \asX$, by length of words. We assume
that each $x \in X$ has degree $1$.

Given an antichain of monomials $W \subset X^{+}$,  the monomial
algebra $A= K\asX/ (W)$ is a particular case of a finitely generated
augmented graded algebra with a set of obstructions $W$, see
\cite{Anick85} and \cite{Anick86}. Here and in the sequel $(W)$ denotes the two-sided ideal
in $K\asX$ generated by $W$. We shall  study monomial algebras
defined by Lyndon words.

This work together with \cite{TGI0612} initiate the study of
algebraic and homological properties \emph{of graded associative
algebras for which the set of obstructions consists of Lyndon
words}. Lyndon words and Lyndon-Shirshov bases are widely used in
the context of Lie algebras and their enveloping algebras, and also
for PI algebras (see for example the celebrated Shirshov theorem of
heights, \cite{Shirshov1}). It will be interesting to explore the remarkable
combinatorial properties of Lyndon words in a more general context
of associative algebras.

Anick studies the class of monomial algebras with finite global
dimension $d < \infty$, \cite{Anick85}. He proves that every such
algebra either i) contains a free subalgebra generated by two
monomials (and therefore has exponential growth);  or ii) $A$ is
finitely presented and has polynomial growth. In the second case he
defines recursively a finite set $N$ of new generators for $A$,
called \emph{atoms}, with $|N|= d$,  and uses the atoms to build a
Poincar\'{e}-Birkhoff-Witt type $K$-basis of $A$ and to describe the
structure of the monomial relations in $W$. Moreover, he proves that
$A$  has the Hilbert series of a (usually nonstandard) graded
polynomial ring:
\[ H_A (t) = \prod_{i = 1}^n \frac{1}{1- t^{e_i}} \]
for some positive integers $e_1, \ldots, e_d$. It is amazing to see
how Anick discovered that his atoms satisfy all good combinatorial
properties of Lyndon words. (Possibly he did not know about Lyndon
words or Lyndon's theorem).

In this paper we  study monomial algebras $A= K \langle X \rangle
/(W)$, where $W$ is an antichain of Lyndon words in $X$. As a
starting point we consider the most general case, when $W$ has
arbitrary cardinality, and no assumptions for finiteness of growth,
or global dimension are made.  In this setting we introduce the set
$N$ of Lyndon atoms, these are the Lyndon words which are normal modulo $(W)$.
(In the context of Lie algebras these are often called standard
Lyndon words, see \cite{LalondeRam}). The set $N$ contains $X$, and,
in general, may be infinite, but exactly the atoms are involved in a
constructive description of both the normal $K$-basis of $A$ and the
set of relations $W,$  so that it is easy to control the growth and
the global dimension of $A$. Using the good combinatorial properties
of Lyndon words we show that
the set $W$ of monomial relations and the set $N$ of Lyndon atoms
are very closely related. We prove that the monomial algebras $A$
defined by Lyndon words have a remarkable property:

\emph{If $A$ has polynomial growth of degree $d$ then $A$ has finite
global dimension $d$ and is standard finitely presented with} $ d-1 \leq |W|
\leq d(d-1)/2$.

In this case the normal $K$- basis of $A$ is $\mathfrak{N} =
\{l_1^{k_{1}}l_2^{k_{2}}\cdots l_d^{k_{d}} \mid k_{i} \geq 0, 1 \leq
i \leq d\}$, where $N = \{l_1, \cdots,l_d\}$ is the set of Lyndon
atoms. Clearly, $\mathfrak{N}$ is a Poincar\'{e}-Birkhoff-Witt type
$K$-basis of $A$, one can consider it as a particular case of
Shirshov basis of height $d$.

Note that in the class  of monomial algebras defined by Lyndon words
our result complements a result by Anick
 which states that if a monomial algebra
$A$ has a finite global dimension $d$ and does not contain
two-generated free subalgebras, then $A$ has polynomial growth of
degree $d$,  \cite{Anick85}, Theorem 6, but the proof of our results
is independent of this theorem of Anick.

A natural question arises: whether our monomial algebras deform to
Artin-Schelter regular algebras. We  find a class of monomial
algebras, Fibonacci-Lyndon algebras $F_n$,  $n \geq 2$, which are
extremal in the class of monomial algebras defined by Lyndon words.
Each $F_n$ has global dimension $n$ and polynomial growth of degree
$n$ and is uniquely determined up to isomorphism.  The algebras
$F_n$ are generated by two variables, hence they are
$\hele^2$-graded. While $F_n$, with $n \leq 5$,  has
$\hele^2$-graded Artin-Schelter deformations, \cite{FlVa}, we show that this is
not the case for $F_6$. However we do not exclude that it may have singly graded such deformations.

One of our goals in the paper is to read off the properties of $A$
directly from its presentation,  and, when this is possible,
independently of Anick's results. What we really use is his notion
of $n$-chains and his purely combinatorial condition in terms of
$n$-chains, necessary and sufficient for finite global dimension,
see Fact \ref{fact_gldim} extracted from \cite{Anick85},  Theorem 4.
Note that all concrete monomial algebras $A$ with i) finite global
dimension and ii) polynomial growth given as examples in
\cite{Anick85} are defined by antichains of Lyndon words. So it is
natural to ask: is it true that if  a monomial algebra $A$ satisfies
i) and ii),  then the set of defining relations $W$ consists of Lyndon
words, w.r.t. appropriate enumeration of the generating set $X$. The
answer is affirmative if $A$ is a quadratic algebra, i.e. $W$
consists of monomials of length $2$, see \cite{TGI0512}, Theorem
1.1. In Section \ref{motivation&background} we give an example of an
algebra with three generators, satisfying i) and ii) and such that
$W$ is not a set of Lyndon words, w.r.t. any ordering of $X$.

The paper is organized as follows. In Section \ref{def&results} we
give basic notions and state our main results, Theorems A and B. In Section
\ref{motivation&background} we give some background and
 motivation. In Section \ref{thenormabasis}
 we prove
some  results about Lyndon words, essential for the paper, we use
these and Lyndon's theorem to show that the normal $K$-basis of $A$
is built out of its \emph{Lyndon atoms} $N$ and prove Theorem A.
 In Section \ref{Sec_determinimgPolyGrowth} we investigate the close relations
 between the set $W$ of
defining Lyndon words, and the set $N$ of Lyndon atoms. In Section
\ref{GlobalDimension}  we find some combinatorial properties of
$n$-chains,
we show that the algebra has finite global dimension whenever the
set  $W$ is finite and prove Theorem B.
 In Section \ref{FibonacciAlgebra}
we define and study the Fibonacci-Lyndon algebras $F_n$, and in
 Section \ref{FibASSec} we show that $F_6$ has no
deformation which is a bigraded Artin-Schelter algebra.

\section{Definitions and results}
\label{def&results}

As usual, $X^\ast$ and $X^{+}$ denote, respectively, the free
monoid, and the free semigroup generated by $X$, ($X^{+}= X^*
-\{1\}$).

Consider the partial ordering on the set $X^{+}$ defined as:
$a\sqsubset b$ \emph{iff} $a$ is a proper subword (segment) of $b$,
i.e. $b = uav$, $|b|> |a|$, but $u = 1$, or $v = 1$ is possible. In
the case when $b = av, a, v  \in X^{+} $,   $a$ is called \emph{a
proper left factor (segment) of} $b$. \emph{Proper right factors}
are defined analogously.

Let $W \sus X^+$. If no two elements of $W$ are comparable for this
partial order, $W$ is called an {\it antichain of monomials}. A
monomial $a\in X^{\ast}$ is $W$-\emph{normal} ($W$-\emph{standard})
if $a$ does not contain as a subword any $u \in W$. Denote by
$\mathfrak{N}(W)$  the set of $W$-normal words
\[\mathfrak{N}(W) = \{a \in X^{\ast}\mid a \; \text{is $W$-normal}\}.\]
Note that the set $\mathfrak{N}(W)$ is closed under taking subwords,
Anick calls such a set  \emph{an order ideal of monomials},
\cite[Sec. 1]{Anick86}.

\medskip
Order the alphabet by  $x_1 < x_2 < \cdots < x_g$. The {\it
lexicographic order} $<$ on $X^+$ is defined as follows: For any
$u,v \in X^{+},\;$ $u < v$ \emph{iff } either $u$ is a proper left
factor of $v$,  or \[u = axb, v = ayc\;\text{with}\; x<y,  \;x,y \in
X,  \; a,b,c \in X^*.\]

The following are well-known, see for example \cite{Lothaire}.
\begin{itemize}
\item[L1.] For every  $u \in X^*$ one has $a < b$ \emph{iff} $ua < ub$.
\item[L2.] If $a$ is not a left segment of $b$,  then for all $u, v \in X^{\ast}$ the inequality  $a <
b$ implies
 $au < bv$.
\end{itemize}

So  ``$<$" is a linear ordering on the set $X^{+}$ compatible with
the left multiplication in $X^{+}$.

\begin{remark} Note that
the right multiplication does not necessarily preserve inequalities,
for example $a< ax_2$, but $ax_3
> ax_2x_3$. Furthermore, the decreasing chain condition on monomials is not
satisfied on $(X^{+}, <)$,  for if $x, y \in X, x<y$, one has  $xy>
x^2y> x^3y > \cdots$. \end{remark}

\begin{definition}
\label{Lyndonworddef} \cite{Lothaire}  A nonperiodic word $u \in
X^{+}$ is \emph{a Lyndon word} if it is minimal (with respect to $<
$) in its conjugate class. In other words, $u = ab, a,b\in X^{+}$
implies $u < ba$. The set of Lyndon words in $X^{+}$ will be denoted
by $L$. By definition $X \subset L$.

Given an antichain $W$ of Lyndon words, the set of $W$-\emph{normal
Lyndon words} will be denoted by $N= N(W)$, we shall refer to $N$ as
\emph{the set of Lyndon atoms corresponding to} $W$. By definition
it satisfies
\[N=N(W) = \mathfrak{N}(W) \bigcap L.\]
\end{definition}

We shall study  finitely generated monomial algebras $A = K\langle
X\rangle / (W)$,  where $|X| \geq 2$ and $W$ is a nonempty antichain
of Lyndon words. By convention we shall consider only \emph{minimal
presentations of} $A$, so
\[W\bigcap X = \emptyset \quad \text{and therefore}\; X \subset N \subset\mathfrak{N}.\]
In this case, inspired  by Anick,  \cite{Anick85},  we call  $N = N(W)$  \emph{the set of Lyndon atoms for}
$A$.

Recall that the graded  associative algebra $A = K \asX /(W)$ has
polynomial growth if there is a real number $d$ and a positive
constant $C$ such that for all $n \geq 0$
\[ \dim_K A_n \leq C n^d. \]   The infimum of the possible $d$'s is the {\it
Gelfand-Kirillov} dimension of $A$.

The main results of the paper are the following two theorems which
are stated under the the same hypothesis:

\emph{Assume that $A = K \asX /(W)$ is a monomial algebra, where $W$ is an
antichain of Lyndon words of arbitrary cardinality, $N=N(W)$ is
the set of Lyndon atoms, and $\mathfrak{N}$ is the set of normal words modulo $(W)$.}

\begin{theoremA}
\label{MonLynTheorem}
\begin{enumerate}
\item
 \label{MonLynTheorem11}
 The set
\begin{equation}\label{normalbasiseq}
 \{ l_1^{k_1} l_2^{k_2} \cdots  l_s^{k_s}\mid s \geq
1, \;   \{l_1 > l_2 > \cdots >l_s\}\subseteq N,  \;  k_i \geq 0,\; 1
\leq i\leq s \}.
\end{equation}
is a $K$-basis of $A$. It coincides with the set of normal words $\mathfrak{N}$.

\item
 \label{MonLynTheorem3}
 The following conditions are equivalent:
\begin{itemize}
 \item[(i)] $N$ is a finite
set;
\item[(ii)] $A$ has polynomial growth;
 \item[(iii)] A is a PI algebra;
 \item[(iv)] A can be embedded in a matrix ring over $K$.
\end{itemize}
\end{enumerate}
In this case  $N = \{l_1 > l_2 > \cdots > l_d\}$,  $A$ has a
Poincar\'{e}-Birkhoff-Witt type $K$-basis $\mathfrak{N} =
\{l_1^{\alpha_{1}}l_2^{\alpha_{2}}\cdots l_d^{\alpha_{d}} \mid
\alpha_{i} \geq 0,  1 \leq i \leq d\}$, so $GK \dim A = d$ and its
Hilbert series  is:
\[
H_A(t) = \prod_{1 \leq i \leq d} \frac{1}{(1-t^{|l_i|})}.
\]
\end{theoremA}

\begin{theoremB}
\label{theoremb}
\begin{enumerate}
\item
 \label{MonLynTheorem1}
 $A = K \asX /(W)$ is a standard finite presentation  \emph{iff} $W$ is
finite.
\item
 \label{MonLynTheorem2}
Suppose $W$ is a finite set of order $|W|=r$, and  $m$ is the
maximal length of words in $W$. Then:
\begin{itemize}\item[(i)] the
global dimension of $A$ is finite and equals at most $r+1$;
\item[(ii)] the algebra $A$ has either polynomial growth or it
contains a free subalgebra generated by two monomials; \item[(iii)]
$A$ has polynomial growth \emph{iff } every word $l$ in $N$ has
length $|l| \leq m-1$.
\end{itemize}
\item
\label{MonLynTheorem4} Suppose
 $A$ has polynomial growth of degree $d$. Then $A$ has finite global
 dimension $d$, and $W$ is of finite order with
  \[d-1 \leq |W|
\leq d(d-1)/2,\] so $A$ is standard finitely presented. Furthermore, the following conditions are
equivalent:
\begin{itemize}\item[(i)]
$|W| = d(d-1)/2$;
\item[(ii)]
$W = \{x_i x_j \mid 1 \leq i < j \leq d\}$;
\item[(iii)]
$N= X$.
\end{itemize}
\end{enumerate}
\end{theoremB}

\begin{corollary}
\label{corThB} $A$ has polynomial growth of degree $d$ if and only
if $A$ has global dimension $d$ and does not contain a free
subalgebra generated by two monomials.
\end{corollary}

\section{Background and Motivation}
\label{motivation&background}
 It was shown by the first author,
\cite{TGI0512}, Theorem 1.1, that an arbitrary finitely presented
monomial algebra $A^0  = K\langle x_1 \cdots, x_n \rangle
/(W)$ with quadratic monomial relations $W$ has polynomial
growth and finite global dimension if and only if there is a
(possibly new) enumeration of the generating set $X =\{x_1 \cdots,
x_n \}$,  so that  $X = \{y_1>y_2> \cdots > y_n \}$ and $W =
\{y_jy_i\mid 1 \leq i \leq j \leq n)$,  that is $W$ consists of all
Lyndon words of length $2$ (w.r.t the new ordering). We believe it
is interesting to know that assuming only quadratic monomial
relations $W$, but no restrictions of their shape, or number,
together with certain algebraic properties  (see the equivalent
conditions (1) $\cdots$ (6) below) lead to exactly $\binom{n}{2}$
defining relations $W$, each of which is a Lyndon word of length
$2$.  For convenience of the reader we give the precise result
which, of course, agrees with the general results of this paper.

\begin{theorem} \cite{TGI0512}
\label{monomialalgebrath} Let $A^0  = K\langle x_1 \cdots, x_n
\rangle /(W)$  be a quadratic monomial algebra. The
following conditions are equivalent:
\begin{enumerate}
\item
 $A^0$ has finite global
dimension and polynomial growth. \item  $A^0$ has finite global
dimension and $|W|= \binom{n}{2}$. \item  $A^0$ has polynomial
growth, $W$ contains no square $x_i^2$, and $|W|=
\binom{n}{2}$. \item The Hilbert series of $A^0$ is
\[
H_{A^0}(z)= \frac{1}{(1-z)^n}.
\]
\item
 There is a permutation $y_1, \cdots, y_n$ of $x_1, \cdots, x_n$ such
 that the set
\[
 \mathfrak{N}= \{y_1^{{\alpha}_1}\cdots y_n^{{\alpha}_n}\mid {\alpha}_i \geq 0, \; 1 \leq i \leq n\}.
 \]
 is a $K$-basis of $A^0$.
\item There is a permutation $y_1, \cdots, y_n$ of $x_1, \cdots, x_n,$
such that
\[W= \{y_jy_i \mid \; 1 \leq i < j \leq n\}.
 \]
 \end{enumerate}
\end{theorem}
One can consider Corollary \ref{corThB} as a generalization of  the
equivalence of conditions (2) and (3) for Lyndon-type relations of
arbitrarily high degrees.

It is shown in \cite{TGI0512}, Theorem 1.2. that each of the
monomial algebras $A^0$ as above share the same obstruction set $W$
with various noncommutative quadratic algebras with binomial
relations which are Artin-Schelter regular,  produce solutions of
the Yang-Baxter equation, and have all good properties of the ring
of commutative polynomials, like being Koszul and Noetherian
domains.

\begin{example}
Let  $A = K \langle x,y,z \rangle /(xz, zy, xxy, xyy, zxy)$. Then
$\omega= xzxyy$ is a $3$-chain, but there are no $4$-chains on $W$
so $A$ has global dimension $4$. One uses the Ufnarovski graph
$\Gamma(A)$, \cite{Ufnarovski82}, (see also
Section \ref{Sec_determinimgPolyGrowth}) to verify that $A$ has
polynomial growth of degree $4$, the cycles $\Gamma(A)$ correspond
to the atoms (in the sense of Anick). In this case the  atoms are $y
> xy > z
> x$, and the normal $K$-basis of $A$ is the set
\[\mathfrak{N}=  \{y^{\alpha_1}(xy)^{\alpha_2}z^{\alpha_3}x^{\alpha_4}\mid \alpha_i \geq 0, 1 \leq i \leq 4\}\]
An easy computation verifies that there is no ordering $<$ on the
alphabet $X =\{x,y,z\} $ such that each of the monomial relations is
a Lyndon word w.r.t. $<$.
\end{example}

One can extract from the proof of Theorem 6 in \cite{Anick85} the
following.
\begin{remark}
Let $A$  be a monomial algebra with global dimension $d$ and
polynomial growth. Suppose $Y = \{y_1,y_2, \cdots, y_d\}$ is the set
of its atoms (in the sense of Anick) enumerated
according Anick's total
order ``$\rightarrow$": $y_d \rightarrow y_{d-1}\rightarrow \cdots
\rightarrow y_2\rightarrow y_1$, see \cite{Anick82}. Then $A$ has a $K$-basis
\begin{equation}
\label{newpres_eq} \mathfrak{N} = \{y_1^{k_1}y_2^{k_2} \cdots
y_d^{k_d}\mid k_j \geq 0, 1 \leq j \leq d\}. \end{equation}
 This induces a new presentation of $A$ in terms of a new
generating set $Y$ and new relations $W_0$:
\[A  \simeq K \langle Y\rangle /(W_0),\]
where $W_0 =\{y_jy_i \mid 1 \leq i < j \leq d\}$. The ordering is
"reverse" to the enumeration, i.e.  $y_j\rightarrow y_i$ \emph{iff}
$j>i$, so the new relations are Lyndon words in the alphabet $Y$,
with total ordering $\rightarrow$. However, whenever the original
set of relations $W$ contains a monomial of degree $> 2$, some of
the new generators (atoms) have degree $> 1$.
\end{remark}

\section{The normal bases  of algebras defined by Lyndon words}
\label{thenormabasis}

In this section we prove some results on Lyndon words which are
essential for the paper. We then describe the normal basis of an
algebra defined by Lyndon words, and prove Theorem A.

We start with some basic facts about Lyndon words, our main
reference is \cite{Lothaire},  Section 5.1., and \cite{Lothaire02},
Section 11.5. (for Lyndon's theorem). As usual, $L$ denotes the set of all
Lyndon words in the alphabet $X$.

\begin{fact}
\label{fact_lyndonwords}
\begin{enumerate}
\item
\label{fact_lyndonwords2} A word $l$ is a Lyndon word if and only if
$l < b$ for any proper right segment $b$ of $l$.
\item
\label{fact_lyndonwords3} For all $w \in L$, the equality $w = ab$,
with $a,b\in X^{+}$, implies $a < w < b.$
\item
\label{fact_lyndonwords1} If $a<b$ are Lyndon words, then $ab$ is a
Lyndon word, so $a < ab < b$.
\item
\label{fact_lyndonwords4} If $b$ is the longest proper right segment
of $l$ which is a Lyndon word, then $l = ab$, where $a$ is a Lyndon
word. This is called \emph{the  standard factorization of $l$} and
denoted as $(a,b)$.
\item
\label{fact_lyndonwords5} If $(a,b)$ is a standard factorization of
a Lyndon word and $c$ is a Lyndon word with $ab < c \leq b$, then
$abc$ is a Lyndon word with standard factorization $(ab, c)$.
\item
\label{fact_lyndonwords6} (Lyndon's Theorem )
 Any word $w \in X^{+}$ can be written uniquely as a nonincreasing product
$w = l_1
l_2 \cdots l_s$ of Lyndon words.
\end{enumerate}
\end{fact}

\subsection{More results on Lyndon words}  \label{Lemmas}
Next we prove some technical results on Lyndon words (in general
context).

\begin{notation}
\label{overlapnotation} For monomials $a, b\in X^{+}$ we shall write
$\;\overbrace{a,b}\;$ if  $a = uv$, $b = vw$, where $u,w \in
X^{\ast}$, $v, uw \in X^{+}$ ($b = aw,$ or $a = ub$ is possible). In
the case when $u,v,w \in X^{+}$ we say that $a$ and $b$
\emph{overlap}.

We shall write $\overbrace{a, \omega, b}$ if  $\omega \sqsubset ab,$
and $\omega$ overlaps with both $a$ and $b$, so that
$\overbrace{a,\omega}$ and $\overbrace{\omega, b}.$ \end{notation}

\begin{lemma}
\label{MonLynLemOverlap}
\begin{enumerate}
\item
\label{MonLynLemOverlap0} If  $v$ is a proper right segment of $l
\in L$, then $v$  is not a left segment of $l$.
 In other words a monomial of
the  shape $l= va= bv$, $a,b, v \in X^{+}$ can not be a Lyndon word.
\item
\label{MonLynLemOverlap01} If  $v$ is a proper right segment of $l = uv
\in L$, then
\begin{equation}
\label{MonLynLem01eq}
lw = uvw  < vw \quad \text{for all}\; w \in X^{\ast}.
\end{equation}
\item
\label{MonLynLemOverlap1} Let $uv$ and $vw$ be Lyndon words. Then
$uvw$ is a Lyndon word.
\item
\label{MonLynLemOverlap2} If $a, b \in L$ and
 $\;\overbrace{a,b}\;$, then $a<b$.
\item
\label{MonLynLemOverlap3} Suppose that $a, b, w \in L$, with
$\overbrace{a, w, b}$. Then $a < w < b$.
\end{enumerate}
\end{lemma}
\begin{proof}
\ref{MonLynLemOverlap0}. Let $w \in L$, and suppose   $v$ is  a
proper right segment of $w$. Fact \ref{fact_lyndonwords} implies
$w <  v$. If we assume that $v$ is
also a proper left segment of $w$, then one has $v<w$, and
therefore  $v<w<v$, which is impossible.

\ref{MonLynLemOverlap01}. Let  $v$ be a proper right segment of $l
\in L$, then $l =uv < v$,  by Fact \ref{fact_lyndonwords}.
Moreover, $l$ is not in  $vX^{+}$, by part \ref{MonLynLemOverlap0}, so L2 implies
(\ref{MonLynLem01eq}).

 \ref{MonLynLemOverlap1}.
Assume that $uv, vw \in L$, and note first that since $v$ is  a proper right segment of the Lyndon word $l = uv$,
(\ref{MonLynLem01eq}) is in force for all $w\in X^{+}$.

We have to show that $uvw$ is a Lyndon word, so by Fact
\ref{fact_lyndonwords} (\ref{fact_lyndonwords2}) it will be enough
to verify that $uvw <b$ holds whenever $b$ is a proper right segment
of $uvw$. Three cases are possible: i) $b$ is a proper right segment
of $vw$; ii) $b= vw$; iii) $b= cvw,$ where $c \in X^{+}$ is a proper
right segment of $u$. Assume (i) holds. Fact \ref{fact_lyndonwords}
implies $vw < b$ which together with (\ref{MonLynLem01eq}) implies
$uvw < vw < b$. In case ii)  the relation (\ref{MonLynLem01eq})
gives straightforwardly $uvw < b =vw$. Assume iii) holds. The
monomial  $cv$ is a proper right segment of the Lyndon word $uv$,
therefore part (\ref{MonLynLemOverlap01}) implies
\begin{equation}
\label{lemmaeq1} uvw < cvw = b, \quad \quad  \forall \; w\in X^{\ast}.
\end{equation}
We have verified part (\ref{MonLynLemOverlap1}).

\ref{MonLynLemOverlap2}. By assumption  $a=uv\in L$, and $b = vw\in
L$, where $v, uw \in X^{+}$. If  $u=1$, then $w \neq 1$, so $a=v$ is
a proper left segment of $b$ and therefore $a < b$. Similarly, if $w
=1, b =v$, then $u \neq 1$, $a = uv = ub\in L$, hence $a < b$. If
$u, w \in X^{+}$, then by part (\ref{MonLynLemOverlap1}) $uvw$ is a
Lyndon word with a proper left segment $a$ and a proper right
segment $b$, hence $a < uvw < b$.

\ref{MonLynLemOverlap3}. By assumption $a, b, w \in L$ and
$\;\overbrace{a, w, b}\;$, hence (by definition)  $\overbrace{a, w}$
and $\overbrace{w, b}$ which, by part (\ref{MonLynLemOverlap2}),
implies $a< w$, and $w < b$.
\end{proof}

\begin{lemma}
\label{MonLynLemOverlap_Cor}
\begin{enumerate}
\item
\label{MonLynLemOverlap4} Let $a < b$ be Lyndon words. Then $a^kb^l$
are Lyndon words for all $k,l \geq 1$.
\item
\label{MonLynLemOverlap5} If $l = ab$ is the standard factorization
of the Lyndon word $l$, then  the standard factorization of $ab^k$ is
$(ab^{k-1}, b)$.
\end{enumerate}
\end{lemma}

\begin{proof}
\ref{MonLynLemOverlap4}. We use induction on $k$ and $l$. By
 Fact \ref{fact_lyndonwords} (\ref{fact_lyndonwords1}), $ab$ is a Lyndon word. Suppose some $a^kb^l$ is a
Lyndon word with $k,l \geq 1$. By Fact \ref{fact_lyndonwords}, part
(\ref{fact_lyndonwords3}), the product  $a(a^{k}b^l) = a^{k+1}b^l$ of
the Lyndon words  $a < a^k b^l$ is a Lyndon word. Similarly,  $a^k
b^l < b$  are Lyndon words, so $a^k b^{l+1}$ is also a Lyndon word.

\ref{MonLynLemOverlap5}. This follows by induction and by Fact
\ref{fact_lyndonwords} (\ref{fact_lyndonwords5}).
\end{proof}

\begin{lemma} \label{MonLynLemFactorization}
Let $l_1 \geq l_2  \geq \cdots \geq l_s$ be Lyndon words, $s \geq
2$. If a Lyndon word $u$ is a subword of $l_1 l_2 \cdots l_s$, then
$u$ is a subword of $l_i$, for some $i$, $1 \leq i \leq s.$
\end{lemma}

\begin{proof}
Let $u\in L$ be  a subword of $l_1l_2\cdots l_s$ and assume that $u$
is not a subword of $l_i$ for $1 \leq i \leq s.$  Then there are
overlaps $\overbrace{l_i, u, l_j}$ for some $1 \leq i<j \leq s$,
hence, by Lemma \ref{MonLynLemOverlap} (\ref{MonLynLemOverlap3}) one
has $l_i < u < l_j$, which contradicts the hypothesis.
\end{proof}

\begin{corollary} \label{nicecorollary}
Every Lyndon word $a$ of length $\geq 2$ contains a subword of the
form $x_ix_j$, $1 \leq i < j\leq g$. \end{corollary}

\begin{proof} Let $a\in L$ with $|a|\geq 2$, and assume, on the
contrary,  that $a$ does not contain any  subword $x_ix_j$, where $1
\leq i < j \leq g$. Then  $a = y_1 y_2 \cdots y_s$, where $y_1 \geq
y_2 \geq \cdots \geq y_s$ are in $X$. By definition $X \subset L$,
so the Lyndon word   $a$ is a (non proper) subword of the product
$y_1 y_2 \cdots y_s$ of non increasing Lyndon words $y_i\in L$.
Lemma \ref{MonLynLemFactorization} implies that $a$ is a subword of
some $y_i$, which is impossible.
\end{proof}

\subsection{The normal basis of $A$.}
In this subsection, as usual,  each of the sets $W$ and $N= N(W)$
may have arbitrary cardinality.
  Lemma \ref{MonLynLemFactorization} implies
straightforwardly the following.
\begin{lemma}
\label{normalbasislemma} Let $W$ be an antichain of Lyndon words
and let $N= N(W)$ be the corresponding set of W-normal Lyndon atoms
in $X^{+}$. Let $\mathfrak{N}$ be the set of all words in $X^{+}$
which are normal modulo the ideal $(W)$. Assume that $N$ contains the Lyndon words
$l_1
> l_2 > \cdots > l_s$.
Then $\mathfrak{N}$ contains the set
\begin{equation}\label{setT}
T(l_1,\cdots, l_s) = \{w \in X^{\ast}\mid  w = l_1^{k_1} l_2^{k_2}
\cdots l_s^{k_s}, k_i \geq 0\}.
\end{equation}
\end{lemma}

\begin{proposition}
\label{normalbasisprop} Suppose  $A= K\asX/ (W)$ is a monomial algebra,
where $W$ is an antichain of Lyndon words  of arbitrary cardinality.
Let $N= N(W)$ be the corresponding set of $W$-normal Lyndon atoms in
$X^{+}$, and let $\mathfrak{N}$ be the set of normal words modulo $(W)$.
Then   $\mathfrak{N}$
 coincides with the set given in (\ref{normalbasiseq}), and is a $K$-basis of $A$.
\end{proposition}
\begin{proof}
It is well-known in the theory of non-commutative Groebner bases that the set of normal monomials
$\mathfrak{N}$ is a $K$-basis of $A$.
Clearly, the set given in (\ref{normalbasiseq}) is the union
\[T = \bigcup _{\begin{array}{c}
s \geq 1\\ \{l_1 > l_2 > \cdots > l_s\}\subseteq N \end{array}} T(l_1,
\cdots, l_s). \]
We shall show that $\mathfrak{N} = T.$
Let $u \in \mathfrak{N}- \{1\}$. Clearly $u \in X^{+}$,
hence by Lyndon's Theorem (see Fact \ref{fact_lyndonwords})
it can be written uniquely as a
product $u = u_1^{k_1}u_2^ {k_2}\cdots u_s^{k_s}$ of Lyndon words,
where $s \geq 1$, $u_1>u_2>\cdots
>u_s, k_i \geq 1, 1 \leq i \leq s$. As a subword of a normal word,
each $u_i$ in this product is also normal, so $u_1,u_2, \cdots, u_s
\in N$. Therefore $u \in T(u_1,\cdots, u_s)$, see (\ref{setT}). We
have shown the inclusion
$\mathfrak{N} \subseteq T$. The reverse inclusion follows from Lemma
\ref{normalbasislemma}.
\end{proof}

\begin{proof}[Proof of Theorem A]
Part
(\ref{MonLynTheorem11}) follows by Proposition \ref{normalbasisprop}

 (\ref{MonLynTheorem3}).
First we show the implication $|N| < \infty\Longrightarrow
\emph{GK}\dim A =|N|$. Assume that $N$ has finite order $d$, so $N =
\{l_1
> l_2 > \cdots > l_d\} \subset L$, and by part (\ref{MonLynTheorem11})
 the $k$ normal basis of $A$ has the desired
form. The vector spaces isomorphism $A \iso Span_K \mathfrak{N}$
implies that $A$ has polynomial growth of degree $d$ and the Hilbert
series of $A$ is the same as the Hilbert series of a polynomial ring
where the generators have degrees $|l_i|$ for $i = 1, \ldots, d$.
This gives the stated form of the Hilbert series.

Next we show that, conversely, $\emph{GK}\dim A = d < \infty
\Longrightarrow |N|= d$. Suppose that $\emph{GK}\dim A = d$.  If we
assume that for some $s> d$, $N$ contains the set of Lyndon atoms
$\{l_1> l_2 > \cdots > l_s\}$, then by Lemma \ref{normalbasislemma}
the normal $k$-basis $\mathfrak{N}$ of $A$ contains the set $T$
given in (\ref{setT}). This implies that $\emph{GK}\dim A \geq s >
d$, a contradiction. Therefore $N$ is a finite set with $|N| \leq
d.$ It follows from the first implication  that $\emph{GK}\dim A =
|N|.$ This gives the equivalence of (i) and (ii). It is proven in a
more general context (no restriction on the shape of $W$) that for a
finitely presented monomial algebra $A$ conditions (ii), (iii), and
(iv) are equivalent, see \cite{borisenko85}. Part
(\ref{MonLynTheorem3}) has been proved.
\end{proof}

\section{Determining polynomial growth}
\label{Sec_determinimgPolyGrowth}
\subsection{Relations between $W$ and $N(W)$}

We have seen that each antichain $W$ of Lyndon monomials determines
uniquely a set $N=N(W)\subset L$, we refer to  it as \emph{the set of Lyndon
atoms corresponding to} $W$. It satisfies  the following
conditions:
\[\begin{array}{l}
\text{\textbf{C1.}}\quad X \subseteq N.\\
\text{\textbf{C2.}} \quad \forall v \in L , \forall u \in N,   v
\sqsubseteq u \Longrightarrow v
\in N.\\
\text{\textbf{C3.}}\quad u \in N \Longleftrightarrow u \in L \;
\;\text{and}\;\; u \notin (W).
\end{array}
\]

Conversely, each set $N$ of Lyndon words satisfying conditions
\textbf{C1} and \textbf{C2} determines uniquely an antichain of
Lyndon monomials $W= W(N)$, such that condition \textbf{C3} holds,
and $N$ is exactly the set of Lyndon atoms corresponding to $W$.
Indeed, let $C = L - N$ be the complement of $N$ in $L$, and let
$W=W(N)$ be the antichain of all minimal w.r.t. $\sqsubset$ elements
in $C$. Then one has $N = N(W(N))$. We shall
 refer to $W = W(N)$  as \emph{the antichain of
Lyndon words corresponding to $N$}. Proposition \ref{N-WProp0} below
gives some of the close relations between the sets $W$ and $N(W)$ on
set-theoretic level.    Our previous discussion implies
straightforwardly parts (\ref{N-WProp01}), (\ref{N-WProp02}) and the first statement in part (\ref{N-WProp30}).
The remaining parts are extracted from Theorems A and B and are
given only for completeness.
\begin{proposition}
\label{N-WProp0}
\begin{enumerate}
\item
\label{N-WProp01} There exists a one-to-one correspondence between
the set $\mathbb{W}$ of all antichains $W$ of Lyndon words with
$X\bigcap W = \emptyset$ and the set $\mathbb{N}$ consisting of all
sets $N$ of Lyndon words satisfying conditions \textbf{C1} and
\textbf{C2}. In notation as above this correspondence is defined as
\[ \begin{array}{lll}
\phi:& \mathbb{W} \longrightarrow \mathbb{N} \quad  &W \mapsto N(W)\\
\phi^{-1}:&  \mathbb{N} \longrightarrow \mathbb{W}\quad  &N \mapsto W(N).
\end{array}\]
\item\label{N-WProp02}
There are equalities
\[N(W(N)) = N; \quad W(N(W))=W,\]
and each pair $(N= N(W), W)$ (respectively $(N, W=W(N))$  obtained
via this correspondence satisfies condition \textbf{C3}.
\item
\label{N-WProp023}
If  $ N \in \mathbb{N}$ is a finite set of order $d$, then the corresponding antichain $W = W(N)$ is also finite
with $|W| \leq d(d-1)/2$.
\item
\label{N-WProp04} Each finite antichain $W \in \mathbb{W}$
determines a monomial algebra $A = K \asX /(W)$ of finite global
dimension, $gl \dim A \leq |W|+1$  .
\item
\label{N-WProp30} Each $N\in \mathbb{N}$ determines uniquely a
monomial algebra  $A = K \asX /(W)$, with a set of defining relation $W = W(N)$ and a set
of Lyndon atoms precisely $N$.
The algebra $A$ has polynomial growth of degree $d$ \emph{iff} $|N|=d$.
\end{enumerate}
\end{proposition}

We shall need the following Lemma, extracted from a more general result in \cite{TGI0612}.
\begin{lemma}
\label{TGIlemma} \cite{TGI0612}  Let $W$ be an antichain of Lyndon words, and assume
$N= \{l_1 < l_2< \cdots < l_d\}$ has finite order $d$. Then there is an inclusion of sets:
\[\{l_il_{i+1}\mid 1 \leq i < j \leq d\} \subseteq W.\]
In particular, $d-1 \leq |W|.$
\end{lemma}

The following theorem gives some of the intimate relations between
 $W$ and $N(W)$ on the level of words.

\begin{theorem}
\label{N_W_Prop} Let $W$ be an antichain of Lyndon words, let $N=
N(W)$ be the corresponding set of Lyndon atoms.
\begin{enumerate}
\item
\label{N_W_Prop1} If $u$ is a proper Lyndon subword of some $w \in
W$, then $u$ is a Lyndon atom, so $u \in N$.
\item
\label{N_W_Prop21} Every word  $w\in W$ factors as $uv$, where $u <
v \in N$.
\item
\label{N_W_Prop3}
If $N= \{l_1 < l_2 <\cdots < l_d\}$ has finite order $d$, then $W$ is also finite
with $d-1 \leq |W| \leq d(d-1)/2$, and there are inclusions of sets:
\begin{equation}
\label{equationN_W_Prop3}
\{l_il_{i+1}\mid 1 \leq i <  d\} \subseteq W \subseteq \{l_il_{j}\mid 1 \leq i < j \leq d\}.
\end{equation}
Moreover, if $s$ is the
maximal length of words in $N$, then  each $w\in W$ has length $|w|\leq 2s$.
\item
\label{N_W_Prop4} Assume  $W$ is finite and let $m$ be the maximal
length of words in $W$. Let $N^{(m-1)}$ be the set of all Lyndon
atoms $u$ of length $\leq m-1$.  The following conditions are
equivalent:
\begin{enumerate}
\item
\label{N_W_Prop4a}
 $N$ is finite;
\item
\label{N_W_Prop4b} every word $l \in N$ has length $|l|\leq m-1$,
that is $N = N^{(m-1)}$.
\item
\label{N_W_Prop4c} Every word  $ab$, where $a, b \in N^{(m-1)}$ with
$a<b$ and $|ab| \geq m$ contains as a subword some $w \in W$.
\end{enumerate}
\end{enumerate}
\end{theorem}

\begin{proof}
Part (\ref{N_W_Prop1}) follows straightforwardly from the definition
of an antichain.

(\ref{N_W_Prop21}). Let $w\in W$. As a Lyndon word $w$ has a
standard factorization $w = uv,$ where $u,v \in L$ and $v$ is the
longest proper right Lyndon segment of $w$. By (\ref{N_W_Prop1})
$u$ and $v$ are Lyndon atoms.

(\ref{N_W_Prop3}). Assume now that $N = \{l_1 < l_2 < \cdots < l_d \}$.
Lemma \ref{TGIlemma} implies the left-hand side inclusion in (\ref{equationN_W_Prop3}) and
the inequality
$d-1 \leq |W|$.
Part (\ref{N_W_Prop21})
implies that $W \subseteq \{l_il_j \mid 1 \leq i < j \leq d\}$,
hence $|W| \leq d(d-1)/2$. This also implies that the length of each
$w\in W$ is at most $2s$, where $s$ is the maximal length of a
Lyndon atom.

(\ref{N_W_Prop4}). Suppose $m$ is the maximal length of words in
$W$.

The implications (\ref{N_W_Prop4b}) $\Longrightarrow$
(\ref{N_W_Prop4a}) and (\ref{N_W_Prop4b}) $\Longrightarrow$
(\ref{N_W_Prop4c}) are clear.

(\ref{N_W_Prop4c}) $\Longrightarrow$ (\ref{N_W_Prop4b}). Suppose
every monomial $u= ab,$ where $a, b \in N^{(m-1)}$, $a<b$, and $m
\leq |a|+|b| \leq 2(m-1),$ contains as a subword some $w \in W$. We
claim that every Lyndon word $l$ of length $|l|\geq m$ is in the ideal $(W)$.
Assume the contrary. Let $l \in L$ be of minimal length, such that
$|l|\geq m$, and $l \notin (W)$.  Clearly, $l \in N$ so the Lyndon
words $a,b$ in its standard factorization $l = ab, a<b$ are also
Lyndon atoms. The lengths of $a$ and $b$ satisfy either i) $|a|\leq
m-1$ and $|b|\leq m-1$; or ii) at least one of the monomials $a$ and
$b$ has length $\geq m$. Note that  (i) is impossible, since it contradicts condition
(\ref{N_W_Prop4c}). Suppose (ii) holds. Without loss of generality, we may assume
$|a|\geq m$. We have found a Lyndon word $a \in N$, such that $m
\leq |a| <|l|$ which contradicts the choice of $l$.

(\ref{N_W_Prop4a}) $\Longrightarrow$ (\ref{N_W_Prop4b}) Suppose $N$
is a finite set of order $d$, and let $N =\{l_1 < l_2 < \cdots <
l_d\}$. It is proven in \cite{TGI0612},  that $l_il_{i+1}\in W$ for
all $1 \leq i \leq d$. Therefore $|l_il_{i+1}| \leq m$, which gives
\[|l_i|\leq m-1,\quad 1\leq i \leq d.\]
\end{proof}

\begin{remark} As we have seen, when $N$ is finite of order $d$ the lower and the upper bounds
for the order $|W|$ are exact.
Theorem B implies that in this case the equality $|W| = d(d-1)/2$ determines the set $W$
uniquely and explicitly. In contrast, when $|W| = d-1$, there may be various $W$'s  reaching  this bound.
One example  is the antichain $W$  defining the Fibonacci algebra $F_d$ of Section \ref{FibonacciAlgebra}.
Another example is defined via its set of Lyndon atoms:
\[
N = \{ x < x^{d-2}y < x^{d-3}y < \cdots < xy < y\}
\]
\end{remark}

\subsection{An algorithm to determine polynomial growth}
In the general case of an s.f.p. monomial algebra $A = K \asX /(W)$,
where $W\subset X^{+}$ is a finite antichain of monomials,  one can
use Ufnarovski's graph $\Gamma(A)$ to decide whether the algebra has
polynomial or exponential growth. For convenience of the reader we recall the definition and an important
result.

The Ufnarovski graph $\Gamma =\Gamma(A)$ of normal words is
 \emph{a directed graph} defined as follows. The
vertices of $\Gamma(A)$ are the non-zero words $u$ of $A$ of length
$m-1$, (that is $u \in \mathfrak{N}(W)$), where $m$ is the maximal
length of a word in $W$. There is an arrow $u \longrightarrow v$
\emph{iff} $ux = yv \in \mathfrak{N}(W)$ for some $x,y \in X$.
 A  cyclic route is called \emph{a cycle},
this is a  path beginning and ending at a vertex $u$.
\begin{fact} \cite{Ufnarovski}
 \label{GNfact}
 \begin{enumerate}
 \item For every $k \geq m$ there is a one-to-one correspondence between the
 set
 of normal words of length
$k$ and the set of paths of length $k-m+1$ in the graph $\Gamma$.
The path $y_1\cdots y_{m-1}\longrightarrow y_{2}\cdots y_m  \cdots
\longrightarrow y_{k-m+1}\cdots y_k $ (these are not necessarily
distinct vertices) corresponds to the word $y_1y_2\cdots y_k \in
\mathfrak{N}$ ($y_1, \cdots, y_k\in X$).
\item
$A$ has exponential growth \emph{iff} the graph $\Gamma$ has two
intersecting cycles.
\item
$A$ has polynomial growth
 of degree $d$ \emph{iff}
 $\Gamma$ has no intersecting cyclic routes (cycles) and
  $d$ is the largest number of (oriented) cycles occurring in a
  path of $\Gamma$.
 \end{enumerate}
 \end{fact}

\begin{remark}
Given $X$ and $W$, formally one can decide effectively whether $A = K \asX /(W)$ has
polynomial growth of degree $d$.

Note that this method does not give a sharp
upper bound  for the length of normal monomials
(or equivalently routes in
$\Gamma$) that have to be checked in order to find the growth.
Clearly, the length of a cyclic route in
$\Gamma$ is bounded by the number of its vertices (that is by the
number of normal words of length $m-1$). Translated to words, this method involves
the words in $\mathfrak{N}$ of length $\leq m-1 + |\mathfrak{N}_{m-1}|$, where $\mathfrak{N}_{m-1}$
is the set of all normal words of length $m-1$.
In contrast with the general case, when $W$ is an antichain of Lyndon
words instead of working with general normal words one works only with Lyndon atoms.
Furthermore, there exists a sharp upper bound for the admissible length of normal Lyndon atoms
in order to have polynomial growth. This bound is  $m$, it is \emph{common for all monomial
algebras with sets of defining relations} $W\subset L$ such that $m=
\max\{|w|\mid w \in W\}.$ Here we have to study whether or not there exists an atom $u$ of length
$m \leq u \leq 2m-2$.
More precisely, knowing all atoms of length $\leq m-1$, whose number is say $d$,
 one has only to check all possible products $ab$, where $a < b$, are atoms of length $\leq m-1$,
 and $|ab|\geq m.$ The number of such products is bounded by $d(d-1)/2.$
\end{remark}
\medskip
Condition (\ref{N_W_Prop4c}) of Theorem \ref{N_W_Prop} implies a
simple method to decide whether $A$ has polynomial growth. Consider
the following problem.
\begin{problem}
Given
\[ \begin{array}{ll}
X = \{x_1, \cdots, x_g\}  &\text{ a finite alphabet}\\
W = \{w_1, \cdots w_r\} \subset L &\text{a finite antichain of
Lyndon monomials}\\
m: =  \max_{1 \leq i \leq r} |w_i|&\\
k: = \min_{1 \leq i \leq r} |w_i|.&
\end{array}
\]
\begin{enumerate}
\item
For each $s = 1, \ldots, m-1$, find the set $N_s$ of Lyndon atoms of
length $s$.

Find  $N^{(m-1)}:= \bigcup_{1 \leq s \leq m-1} N_s$.

\item Decide whether the monomial algebra $A = K \asX /(W)$ has
polynomial growth, and if "yes",

\item Find $GK\dim A$, the Gelfand-Kirillov dimension of $A$,
and $gl\dim A$, the global dimension of $A$.
\end{enumerate}
\end{problem}

In the settings of this problem the question of finding the global dimension of $A$  (only in case of polynomial growth)
is answered straightforwardly. In the general case of finitely presented monomial
algebras, or s.f.p. associative algebras, there exist various algorithms, which (implementing Anick`s results)
find the global dimension directly using the reduced Groebner basis, and Anick's resolution,
see  \cite{TGI89}, \cite{Ufnarovski97}, et all.

\begin{method}

\begin{enumerate}

\item
For $1 \leq i \leq k-1$, we set $N_i =\{u \in L \mid
|u| = i\}$.

For $k \leq s \leq m-1$ we find $N_s$  recursively.

Suppose $N_j$ is found for all $1 \leq j \leq s-1$, denote
$N^{(s-1)} = \bigcup_{1 \leq i \leq s-1} N_i$. Then
\begin{align*}
N_s= \{ab \mid & a, b \in N^{(s-1)}, a < b,  |a|+|b|= s, \; \\
&\quad \text{no $w\in W$ is a subword of} \;ab\}\\
N^{(m-1)} = &\bigcup_{1 \leq i \leq m-1} N_i\\
d: = &\;  |N^{(m-1)}|.
\end{align*}

\item
For each pair  $a, b \in N^{(m-1)}$ with $a<b$ and
$|ab| \geq m$ check whether $ab$ has some $w \in W$ as a subword.

If "YES", then $A$ has polynomial growth, proceed to 3.

If, ``NO" (i. e. there exist $a, b \in N^{(m-1)}$ with $a<b$ and
$|ab| \geq m$, such that no $w\in W$ is a segment of $ab$), then $A$
has exponential growth. The process halts.

\item
Set \[GK \dim := d;  \quad gl \dim A := d.\]

\end{enumerate}
\end{method}

\begin{remark}
Note that
in contrast with Anick's proof, see \cite{Anick85}, the special
shape of the elements of $W$ makes it possible to straightforwardly
determine  the so called ``atoms" of $A$, which are difficult to
find explicitly using Anick's result.  In our case these are the
Lyndon atoms.
To find the atoms in the general case  of a finitely presented monomial algebra with polynomial growth,
one can use the graph $\Gamma$:
the atoms correspond to the cycles in $\Gamma$.
\end{remark}

\section{The global dimension}
\label{GlobalDimension}

Given an antichain of monomials $W \subset X^{+}$ the monomial
algebra $A= K\asX/ (W)$ is a particular case of  finitely generated
augmented graded algebras with a set of obstructions $W$, see
\cite{Anick85} and \cite{Anick86}. Anick constructs a resolution, of
the field $K$ considered as an $A$-module, and obtains important
results on algebras with polynomial growth and finite global
dimension, see \cite{Anick86}, and \cite{Anick85}. The "bricks" of
Anick's resolution are the so called \emph{ $n$-chains on $W$}.
Anick's resolution is minimal whenever $A$ is a monomial algebra. We
recall first the  definition of an $n$-chain and a result from
\cite[Sec. 3]{Anick85}.

\begin{definition}
\label{nchaindef}  The set of  $n$-chains on $W$ is
defined recursively. A
 $(-1)$-\emph{chain} is the monomial $1$, a
$0$-\emph{chain} is any element of $X$, an $1$-\emph{chain} is a
word in $W$. An $(n+1)$-\emph{prechain} is a word $w \in X^{+}$,
which can be factored in two different ways $w = uvq=ust$ such that
$t \in W$, $u$ is an $(n-1)$-chain, $uv$ is an $n$-chain, and $s$ is
a proper left segment of $v$. An $(n+1)$-\emph{prechain} is an
$(n+1)$-chain if no proper left segment of it is an $n$-chain. In
this case the monomial $q$ is called \emph{the tail of the
$(n+1)$-chain $w$}.
\end{definition}

The following fact can be extracted from \cite[Theorem 4]{Anick85}.
\begin{fact}
\label{fact_gldim}
 Let  $A = K \asX / (W)$ be a
monomial algebra, where  $X$ is a nonempty set of arbitrary
cardinality, and $W $ is an antichain of monomials in $X^{+}$.
 The global dimension of $A$ is $n$ \emph{iff}
there exists an $(n-1)$-chain  but there are no $n$-chains on $W$.
\end{fact}

One can read off  Anick's definition that every $n$-chain is "built"
out of a string of $n$- monomials from $W$ which overlap
successively in a special way. The lemma below is straightforward.

\begin{lemma}
\label{nchainlemma} Let $W$ be a nonempty antichain of monomials in
$X^{+}$, $n \geq 2$. The word $\omega \in X^{+}$ is an $n$-prechain
if and only if it has a presentation
\begin{equation}
\label{nchaineq} \omega =  v_0 t_1 v_1 t_2 v_2 \cdots t_{n-2}
v_{n-2} t_{n-1} v_{n-1}t_nv_n,
\end{equation}
such  that
\begin{enumerate}
\item
$v_0 \in X$, and $v_i \in X^{+}, t_i \in X^{\ast}, 1 \leq i \leq n$;
\item
\label{nchainlemma2} Each   $u_i= v_i t_i v_{i+1}$  is a word in
$W$, $0 \leq i \leq n-1$, and one has $\overbrace{u_i, u_{i+1}}$ for
all $0 \leq i \leq n-2$.
\item
For  $1 \leq m \leq n-1,$  the left proper segment  \[w_m = v_0 t_1
v_1 t_2 v_2 \cdots v_{m-1} t_mv_m\] of $\omega$ is an $m$-chain on
$W$ with a tail $t_mv_m$ . Furthermore,
\[v_0 t_1 v_1 t_2 v_2 \cdots t_{n-2} v_{n-2} t_{n-1} v_{n-1}\] is
the unique $(n-1)$-chain contained as a left segment of $\omega$.
(The initial letter, $w_0 = v_0\in X$ is a  $0$-chain).
\end{enumerate}
An $n$-prechain is an $n$-chain if no proper left segment of it is
an $n$-prechain.
\end{lemma}
\begin{proposition}
\label{nchainprop} Let $W$ be a nonempty antichain of Lyndon words
(of arbitrary cardinality). Then
\begin{enumerate}
\item
\label{nchainprop1}
 Every $2$-prechain is a Lyndon word.
\item
\label{nchainprop2} Every $n$-chain $\omega$, $n \geq 1$, is a
Lyndon word.
\item
\label{nchainprop3} Suppose  $\omega$ is an $n$-chain, $n \geq 1$.
Let $u_i = v_{i-1}t_iv_i \in W$ ,   $1 \leq i \leq n$ be the Lyndon
monomials involved in its presentation (\ref{nchaineq}). Then
$\overbrace{u_i, u_{i+1}}$, for all $1 \leq i \leq n-1$, so there
are  strict inequalities
\[ u_1 < u_2 < \cdots < u_n.\]
\item
\label{nchainprop4} If $W$ is a set of order $r$, then there are no
$(r+1)$-chains on $W$.
\end{enumerate}
\end{proposition}

\begin{proof}
\ref{nchainprop1}.  Let $\omega$ be a $2$-prechain, then it factors
as $\omega= uvw$, where $u,v,w\in X^{+}$, and $uv, vw \in W
\subseteq L$. Hence $uv$ and $vu$ are Lyndon words, and by Lemma
\ref{MonLynLemOverlap} the product $uvw = \omega$ is also a Lyndon
word.

(\ref{nchainprop2}) We prove that every $n$-chain $\omega$ is a
Lyndon word by induction on $n$. By definition each $1$-chain is an
element of $W$, and therefore it is a Lyndon word. We just proved
that the 2-chains are also Lyndon words. Suppose now that for $2
\leq m \leq n$ every $m$-chain is a Lyndon word, and let $\omega$ be
an $(n+1)$-chain. Then in notation as in Lemma \ref{nchainlemma},
$\omega = (w_{n-1}t_{n}). v_{n} (t_{n+1}v_{n+1})=uvw$,
where $w_{n-1}$ is an $(n-1)$-chain, $u= w_{n-1}t_{n}$,  $v=v_{n}$,  $uv = w_{n-1}t_{n}v_{n}$
is the unique $n$-chain contained as a left segment of $\omega$, and
$vw = v_{n}t_{n+1}v_{n+1} \in W.$ By the inductive assumption the
$n$-chain $uv$ is a Lyndon word, clearly $vw\in W$ is also a Lyndon
word. It follows then from Lemma \ref{MonLynLemOverlap} that $\omega
= uvw$ is a Lyndon word, which proves part (\ref{nchainprop2}) of
the proposition.

(\ref{nchainprop3}) Suppose  $\omega$ is an $n$-chain.  Then each of
the monomials $u_m = v_{m-1}t_mv_m$,    $1 \leq m \leq n$, involved
in its presentation (\ref{nchaineq}) is an element of $W$, so it is a
Lyndon word. Clearly, for $1 \leq i \leq n-1$ one has
$\overbrace{u_i, u_{i+1}}$, so by Lemma \ref{MonLynLemOverlap}
\ref{MonLynLemOverlap2} $u_i<u_{i+1}$. This yields $u_1 < u_2 <
\cdots < u_n$, which proves (\ref{nchainprop3}).

Part (\ref{nchainprop4}) follows straightforwardly from
(\ref{nchainprop3}).
\end{proof}

\begin{proof}[Proof of Theorem B]
(\ref{MonLynTheorem1}). It is well-known that any antichain of
monomials $W$ is a minimal Gr\"{o}bner basis of the ideal $(W)$,
thus $A = K \asX /(W)$ is a standard finite presentation of $A$
\emph{iff} $W$ is finite.

 (\ref{MonLynTheorem2}).
Suppose $|W| = r$. By Proposition \ref{nchainprop} part
(\ref{nchainprop4}) there are no $(r+1)$-chains on $W$, so  Fact
\ref{fact_gldim}  implies that $\emph{gl} \dim A \leq r+1$ which
proves part (i). Part (ii) follows straightforwardly from the
results of Ufnarovski, \cite{Ufnarovski82}.  For convenience of the
reader, only, we shall give a sketch of a proof.
We shall use the results of Ufnarovski recalled in Section \ref{Sec_determinimgPolyGrowth}.

It follows from Fact \ref{GNfact}  that an s.f.p. algebra  $A$ has either polynomial or exponential
growth. Assume that $A$ has exponential growth, so by Fact
\ref{GNfact} the graph $\Gamma$ has two intersecting cycles $C_1$
and $C_2$.  Let $a_1, a_2$, respectively, be the corresponding
normal words. Then every word $u$ in the alphabet $a_1, a_2$
corresponds to a route in $\Gamma$ and therefore $u \in
\mathfrak{N}$. This implies that $A$ contains the free algebra
generated by $a_1$ and $a_2$.

 Part (iii) follows from Theorem \ref{N_W_Prop} (\ref{N_W_Prop4}).

(\ref{MonLynTheorem4}).
Assume  $\emph{GK}\dim A = d$, then, as we
have already shown, $|N| = d$   and therefore, by Theorem
\ref{N_W_Prop} (\ref{N_W_Prop3}), one has $|W| \leq d(d-1)/2$.

A resent result of \cite{TGI0612} verifies (independently of Anick's
results) that if $W$ is an antichain of Lyndon words,  and $|N|= d$,
then  there exists a  $(d-1)$-chain, but there is no $d$-chain on
$W$, and therefore by  Fact \ref{fact_gldim}, $A$ has global
dimension $d$. By Theorem A the order  $|N|=d$, is also  the
GK-dimension of $A$.

We give now  a second (indirect) argument for the equality between the global dimension and GK-dimension of $A$.

By assumption $A$ has polynomial growth,  so $W$ is a finite set, and by
part (\ref{MonLynTheorem2})  $A$ has finite global dimension.
Clearly $A$ does not contain a free subalgebra generated by two monomials,
and therefore by   \cite[Theorem 6]{Anick85} there is an equality  $GK\dim A = gl\dim A$.

Suppose $N = \{l_1 < l_2 < \cdots < l_d\}$.
By Theorem \ref{N_W_Prop}
\[W \subseteq \{ l_il_j \mid 1 \leq i < j\leq d \},\]
so the equality $|W| = d(d-1)/2$ implies an equality of the sets above.

By convention
$X\subseteq N$, and therefore
\[
W_0 = \{x_ix_j |1 \leq i < j \leq g\} \subseteq W.\]
 Note that $W_0$, consists of all Lyndon words of
length $2$. Corollary \ref{nicecorollary} implies that every Lyndon
word $a$ of length $\geq 2$ contains a subword of the form $x_ix_j$,
$1 \leq i < j\leq g$, hence $a \in (W_0)$. Note that $W$ is an
antichain of Lyndon monomials of length $\geq 2$, it follows then
that  $W= W_0$. This implies the equalities $N= X$, $g=d$, so $A$ is
presented as
\begin{equation}
\label{eqA} A = k \langle x_1, \cdots, x_d  \rangle /(W_0),\quad
W_0= \{x_ix_j \mid 1 \leq i < j\leq d\}.\end{equation}

Conversely, if the monomial algebra $A$ is defined by \ref{eqA},
where $g = d,$, then $A$ satisfies the hypothesis of the theorem, the
set of its Lyndon atoms is $N = X,$ and $|W| = d(d-1)/2$.

\end{proof}

Corollary \ref{corThB} is straightforward from Theorem B (3), and Anick's result
\cite[Theorem 6]{Anick85}.

\begin{remark}
The fact that every standard finitely presented graded algebra $A$
has either polynomial or exponential growth is already classical. It
follows
from Ufnarovski's results, \cite{Ufnarovski82}, see also Fact
\ref{GNfact}. As we have seen for finitely presented monomial
algebras exponential growth is equivalent to the existence of a free
subalgebra generated by two monomials, a condition which is, in
general, stronger than having exponential growth.
\end{remark}

\begin{corollary}
\label{monomialalgebrathCOR} Let $A^0  = K\langle x_1
\cdots, x_n  \rangle /(W_0)$  be a monomial algebra, where $W_0
\subseteq X^{+}$ is an arbitrary antichain of monomials. The
following conditions are equivalent:
\begin{enumerate}
\item
\label{theor11cor} The set of monomial relations $W_0$ consists of
Lyndon words,
$A^0$ has polynomial growth of degree $d$, and $|W_0| = d(d-1)/2$.
\item
\label{theor12cor} $A^0$ is a quadratic algebra, i.e.  $W_0$
consists of monomials of length $2$, $n=d$,  and at least one of  conditions
 (1) through (6)  of Theorem \ref{monomialalgebrath} is
satisfied.
\end{enumerate}
In this case all conditions (1), $\cdots$, (6) of Theorem
\ref{monomialalgebrath} hold.
 \end{corollary}

\section{An extremal algebra: The Fibonacci algebra}
\label{FibonacciAlgebra}
\subsection{Fibonacci-Lyndon words}
Consider the alphabet $X = \{x,y\}$.
Define the sequence of {\it Fibonacci-Lyndon words} $\{\fib_n(x,y)\}$
by the initial conditions $\fib_0 = x, \fib_1 = y$ and, then for $n \geq 1$
\begin{equation} \label{FibLigDef}
 \fib_{2n} = \fib_{2n-2}\fib_{2n-1}, \quad \fib_{2n+1} = \fib_{2n} \fib_{2n-1}.
\end{equation}
This give the sequence

\vskip 3mm
\begin{tabular}{c|c|c|c|c|c|c}
 0 & 1 & 2 & 3 & 4 & 5 & 6 \\
\hline
 $x$ & $y$ & $xy$ & $xyy$ & $xyxyy$ & $xyxyyxyy$ & $xyxyyxyxyyxyy$. \\
\end{tabular}

\vskip 3mm \noindent Note that if we let $a$ be $\fib_2(x,y) = xy$
and $b$ be $ \fib_3(x,y) = xyy$, then the Fibonacci-Lyndon word
$\fib_m(a,b) = \fib_{m+2}(x,y)$.

\begin{lemma} The following holds:
\begin{itemize}
\item[a.] The word $\fib_n(x,y)$ is a Lyndon word and its length
is the $n$'th Fibonacci number.
\item[b.] For the lexicographic order we have
\[ \fib_0 < \fib_2 < \cdots < \fib_{2n} < \cdots < \fib_{2n+1} < \cdots < \fib_3
< \fib_1. \]
\end{itemize}
\end{lemma}

\begin{proof}
By induction we see that $\fib_{2n}$ and $\fib_{2n+1}$ are Lyndon
words and their lengths are as stated, which gives part a.

Now the recursive definition (\ref{FibLigDef}) and Fact
\ref{fact_lyndonwords} (\ref{fact_lyndonwords3}), imply that for each
$n \geq 1$ the Fibonacci-Lyndon words satisfy
\begin{equation}
\label{fib_eq1}
\begin{array}{c}
\fib_{2n-2} < \fib_{2n} = \fib_{2n-2}\fib_{2n-1}< \fib_{2n-1}\\
\fib_{2n} < \fib_{2n+1} = \fib_{2n} \fib_{2n-1} < \fib_{2n-1}.
\end{array}
\end{equation}
This straightforwardly proves part (b).
\end{proof}

Let $U$ consist of all Lyndon words $\fib_{2n-2}\fib_{2n}$ and
 $\fib_{2n+1}\fib_{2n-1}$, where $n \geq 1$.

\begin{proposition} A Lyndon word $w$ in $x$ and $y$ is not in the ideal
$(U)$ if and only if it is a Fibonacci-Lyndon word.
\end{proposition}

\begin{proof} We argue by induction
on the length of $w$, the statement clearly holds when the length
is one. Suppose the length of $w$ is $\geq 2$. Note that $\fib_0
\fib_2 = xxy$ and $\fib_3 \fib_1 = xyyy$, so, if a Lyndon word $w$
is not in the ideal $(U)$ we see that $w$ must be a word in $a = xy$
and $b = xyy$. Since $\fib_m(a,b) = \fib_{m+2}(x,y)$ we see that $w$
is not divisible by any of
\[ \fib_{2p-2}(a,b) \fib_{2p}(a,b), \quad \fib_{2p+1}(a,b) \fib_{2p-1}(a,b). \]
Considering the length of $w$ written in terms of $a$ and $b$ and
using induction, we prove that $w = \fib_m(a,b) = \fib_{m+2}(x,y)$ for some $m$.
\end{proof}

\subsection{The extremal algebra}
Let $W_n$ be the antichain of all minimal elements in $U \cup \{ \fib_n(x,y) \}$,
with respect to the divisibility order $\sqsubset$.
This is a finite set, since $\fib_n$ is
a factor of the Fibonacci-Lyndon words later in the sequence.

The
set of  Lyndon atoms with respect to the ideal $(W_n)$ is $N_n = \{
\fib_0, \ldots, \fib_{n-1} \}$, and so we obtain a monomial algebra,
 {\it the  Fibonacci algebra}
\[ F_n = k\asX /(W_n),  \]
whose Hilbert series is
\[ \prod_{i =0}^{n-1} \frac{1}{1-t^{|f_i|}},  \]
where $f_i$ is the $i$'th Fibonacci number.
Clearly,  the global
dimension and the Gelfand-Kirillov dimension of $A$ are both $n$.

This algebra is extremal in the following sense.
\begin{proposition} Let $W$ be a finite set of Lyndon words such that
the corresponding set of Lyndon atoms,
$N(W) = \{ w_0, \ldots, w_{n-1} \}$, is
finite and enumerated according to increasing lengths of $w_p$. Then the
lengths satisfy $|w_p| \leq |\fib_p|$. If we have an equality
for each $p$, $0 \leq p \leq n-1,$
then the algebra $A = K\asX / (W)$ is isomorphic to the Fibonacci
algebra $F_n$.
\end{proposition}

\begin{proof}
By convention, $X$ has at least two elements, and  $X \sus N$.
Clearly, then $|w_i| \leq |\fib_i|$ for $i = 0,1$.
Let $n-1 \leq p \geq 2,$
and let $w_p = lm$ be the standard factorization  of $w_p$.
Then $l,m$ must be in $N$ so we may write $w_p = w_iw_j$
where $i$ and $j$ are distinct integers $ < p$.
If $i < j< p$,
then $i \leq p-2, j \leq p-1$ so by the
inductive assumption, $|w_i| \leq |\fib_{p-2}|$ and $|w_j|  \leq
|\fib_{p-1}|$, and therefore $|w_p| = |w_i| + |w_j|\leq |\fib_{p}|$. The case $i >
j$ is analogous.

Assume now $|w_p| = |\fib_{p}|$, $0 \leq p \leq n-1$. Clearly,
$X$ must consist of two elements $x < y$, and
there are equalities $w_0 = x$ and $w_1 = y$, or the other way around. In any
case we must have $w_2 = xy$. There are now two possibilities  for
$w_3$. It is either $xyy$ or $xxy$. Assume $w_3 = xyy$. We then prove
by induction that $w_p = \fib_p$ for $p = 0, \ldots, n-1$.
Clearly, the standard factorization of $w_{2r}$ is either $w_{2r-1}
w_{2r-2}$ or $w_{2r-2} w_{2r-1}$. Since $\fib_p = w_p$ for $p < 2r$
and given the ordering of the Fibonacci-Lyndon words, the latter
must hold, and so $w_{2r} = \fib_{2r}$. Similarly, we may argue that
$w_{2r+1} = \fib_{2r+1}$.

In the case $w_3 = xxy$ we take $w_0 = y$ and $w_1 = x$. There is an
involution $\tau$ on $k\langle x, y \rangle$ which takes a word
$a_1a_2 \cdots a_r$ and arranges it in the opposite order $a_r
\cdots a_1$. There is also an involution $\iota$ which replaces each
$x$ with $y$ and each $y$ with $x$. It is not difficult to argue
that $w_p = \iota \circ \tau (\fib_p)$, and so $A$ and $F_n$ become
isomorphic via the map $\iota \circ \tau$.
\end{proof}

\section{Fibonacci algebras not deforming
to Artin-Schelter regular algebras}

\label{FibASSec}

It is known that to each Lyndon word $l$ one may associate a Lie monomial (called bracketing of $l$, and denoted by
  $[l]$),
\cite{Lothaire}, Chapter 5, or \cite{Reutenauer},  Chapter 4.
The Lie monomials corresponding to Lyndon words
form a basis for the free Lie algebra, $Lie(X)$, generated by $X$.

To each monomial algebra $A = K\asX/(W)$ defined by an antichain of Lyndon words $W$ we associate canonically
the (associative) algebra $\tilde{A} = K\asX/([W])$ with the same generating set $X$,
and the set $[W]$ of Lie monomials associated with $W$ as defining relations.
(As usual,  a Lie element $[a,b]\in Lie(X)$ is
considered also as an "associative" element $[a,b] = ab-ba \in K\asX$). In this case
the algebra $\tilde{A}$ is an enveloping algebra of the Lie algebra $\mathcal{L}$ generated by $X$ and
with the same set of defining relations
(considered as elements in
$Lie(X)$).
The first question to ask is whether the monomial algebra $A$  and the corresponding
enveloping algebra $\tilde{A} $ share the same $K$-basis, or, equivalently,
the same Hilbert series. This is not so, in general, but
when this holds is further investigated in \cite{TGI0612}.
Enveloping algebras of finite dimensional graded Lie algebras are special cases of
Artin-Schelter regular algebras. It is then natural to ask if our monomial algebras may
deform to algebras in this more general class.

\medskip
An algebra $A = k \oplus A_1 \oplus A_2 \oplus \cdots$ is an
Artin-Schelter regular algebra of dimension $d$ if:

\begin{itemize}
\item $A$ has finite global dimension $d$.
\item $A$ has finite Gelfand-Kirillov dimension.
\item $A$ is Gorenstein, i.e.
\[ \Ext^i_A(K,A) = \begin{cases} 0 & i \neq d \\
                                K(l) & i = d
                   \end{cases}
\]
for some shift $l$.
\end{itemize}

The monomial algebras defined by Lyndon words with a finite
set $N$ of Lyndon atoms have the two first properties. It is therefore natural to
ask if they can be deformed to Artin-Schelter regular algebras.
If $B$ is such a monomial algebra, its Hilbert series is
\[ H_B(t) = \prod_{l \in N} \frac{1}{1-t^{|l|}}. \]
If the resolution of the residue field $K$ of $B$ is
\[ B \vpil \cdots \vpil \oplus_{j \in \hele} B(-j)^{\beta_{ij}} \vpil
\cdots \]
then
\[ H_B(t) ( \sum_{i,j \in \hele} (-1)^i \beta_{ij} t^j) = 1 \]
so
\[  \sum_{i,j \in \hele} (-1)^i \beta_{ij} t^j = \prod_{l \in N} (1-t^{|l|}). \]
Thus, since this polynomial is symmetric up to sign, it is numerically
possible that the monomial algebra $B$ deforms to an algebra with
the Gorenstein property.  Fl{\o}ystad and J.E.Vatne
show that the $\hele^2$-graded
Fibonacci-Lyndon monomial algebras $F_5$ for $n \leq 5$ all deform to
Artin-Schelter regular algebras which are also $\hele^2$-graded, \cite{FlVa}.
For $n \leq 4$ these deformations are enveloping algebras of Lie
algebras but it is not so for $n = 5$. However we have the following.

\begin{proposition} The Fibonacci-Lyndon monomial algebra $F_6$
does not deform to a bigraded Artin-Schelter regular algebra.
\end{proposition}

\noindent{\bf Remark.} There still remains the possibility though that
it might deform to a singly graded Artin-Schelter regular algebra.

\begin{proof}
Since the complete argument involves a lot of computation,
we will give only a sketch for for the last parts of the proof.

\noindent{Part 1.} Let $B$ be the monomial algebra $F_6$. The resolution of its
residue field may be worked out to be (we write the multidegrees
of the generators of the free modules below):
\begin{equation*}
\underset{\scriptsize{\begin{matrix} (0,0) \end{matrix}}}
{B} \vpil
\underset{\scriptsize{\begin{matrix} (1,0)\\ (0,1)
\end{matrix}}} {B^2} \vpil
\underset{\scriptsize{\begin{matrix} (2,1) \\ (1,3) \\ (3,4) \\
(5,8) \\ (4,7)
 \end{matrix}}} {B^5} \vpil
\underset{\scriptsize{\begin{matrix} (2,3) & (4,8) \\
(4,4) & (5,8) \\
(3,5) & (5,9) \\
(5,7) & (6,9) \\
(6,8) & (6,10)
\end{matrix}}}
{B^{10}}
 \vpil
\underset{\scriptsize{\begin{matrix} (4,5) & (6,10) \\
(5,8) & (6,11) \\
(5,9) & (7,9) \\
(6,8) & (7,10) \\
(6,9 &
\end{matrix}}}
{B^{9}}
 \vpil
\underset{\scriptsize{\begin{matrix} (6,9)  \\
(7,10) \\
(7,11) \\
(8,11)
\end{matrix}}}
{B^{4}}
\vpil
\underset{\scriptsize{\begin{matrix} (8,12)
\end{matrix}}}
{B}
\end{equation*}

If $B$ deforms to a bigraded Artin-Schelter regular algebra $A$, its resolution
is obtained by canceling adjacent terms
of the same multidegrees in the above resolution, and it must have
a selfdual form, since $\Tor_i^A(k,k)$ can be computed by taking a resolution
of $K$ either as a left
or as a right module.
The only possibility for the minimal resolution of $A$ is then

\begin{equation} \label{FibLigResAk}
\underset{\scriptsize{\begin{matrix} (0,0) \end{matrix}}}
{A} \vmto{d_0}
\underset{\scriptsize{\begin{matrix} (1,0)\\ (0,1)
\end{matrix}}} {A^2} \vmto{d_1}
\underset{\scriptsize{\begin{matrix} (2,1) \\ (1,3) \\ (3,4) \\
(4,7)
 \end{matrix}}} {A^4} \vmto{d_2}
\underset{\scriptsize{\begin{matrix} (2,3) \\
(4,4) \\
(3,5) \\
(4,8) \\
(5,7) \\
(6,9)
\end{matrix}}}
{A^{6}}
 \vpil
\underset{\scriptsize{\begin{matrix} (4,5) \\
(5,8) \\
(6,11) \\
(7,9) \\
\end{matrix}}}
{A^{4}}
 \vpil
\underset{\scriptsize{\begin{matrix} (7,12)  \\
(8,11)
\end{matrix}}}
{A^{2}}
\vpil
\underset{\scriptsize{\begin{matrix} (8,12)
\end{matrix}}}
{A}
\end{equation}
where $d_0 = [x,y]$.

The differentials here are represented by matrices whose entries are in
$A$. These entries are then of the form $\pi(p)$, where
$\pi: \; K<x,y> \rightarrow   A$ is the natural quotient map, and
$p \in K<x,y>$. By abuse of notation we shall
simply write $p$ for such an entry. Since the composition of successive
differentials is zero, the product of any two successive matrices
will then have entries which are {\it relations} for A, i.e.
they are in the kernel of $\pi$.

In particular, the defining relations of $A$ are given by the
elements of the product matrix $d_0 \cdot d_1$, which have bidegrees
$(2,1), (1,3), (3,4)$ and $(4,7)$.

\medskip
\noindent{Part 2.}
Now look at the subcomplex
\[
\underset{\scriptsize{\begin{matrix} (0,0) \end{matrix}}}
{A} \vmto{d_0}
\underset{\scriptsize{\begin{matrix} (1,0)\\ (0,1)
\end{matrix}}} {A^2} \vmto{d_1^\prime}
\underset{\scriptsize{\begin{matrix} (2,1) \\ (1,3)
 \end{matrix}}} {A^2} \vmto{d_2^\prime}
\underset{\scriptsize{\begin{matrix} (2,3)
\end{matrix}}}
{A.}
\]

After suitable base changes we may assume that
\[ d_1^\prime = \left [
\begin{matrix} xy + \alpha_0yx & y^3 \\
       -\alpha_1 x^2 & \beta_0 xy^2 + \beta_1 yxy + \beta_2 y^2x
\end{matrix} \right ], \;
d_2^\prime = \left[ \begin{matrix} y^2 \\ \gamma x \end{matrix} \right ].
\]

Multiplying $d_0 = [x,y]$ with $d_1^\prime$ we get the two first of the
four defining relations for $A$:
\begin{align}
\label{FibLigEn} &  x^2y + \alpha_0xyx - \alpha_1 yx^2  \\
\label{FibLigTo} & xy^3 + \beta_0yxy^2 + \beta_1 y^2xy + \beta_2 y^3x.
\end{align}

The product of $d_1^\prime$ and $d_2^\prime$ induces the following relations of $A$.

\begin{align}
\label{FibLigTre}
& xy^3 + \alpha_0yxy^2 + \gamma y^3x \\
\label{FibLigFire}
 -\alpha_1  & x^2y^2 + \gamma \beta_0 xy^2x + \gamma \beta_1 yxyx +
\gamma \beta_2 y^2x^2.
\end{align}
So (\ref{FibLigTre}) must be a consequence of (\ref{FibLigTo}) which gives
\[
\beta_1 = 0, \, \beta_0 = \alpha_0, \, \beta_2 = \gamma .\] Then
the relation (\ref{FibLigFire}) becomes
\begin{equation}
\label{FibLigFem}
- \alpha_1 x^2y^2 + \gamma \alpha_0 xy^2x +
\gamma^2 y^2x^2.
\end{equation}

The relation (\ref{FibLigFem}) must be a linear combination of the following expressions
 obtained by multiplying
(\ref{FibLigEn}) with $y$ on the left and on the right.

\begin{align*}
x^2y^2 + \alpha_0 xyxy - & \alpha_1 yx^2y  \\
& yx^2y + \alpha_0yxyx - \alpha_1y^2x^.
\end{align*}

If  (\ref{FibLigFem}) is nonzero, then  this linear combination is
also nonzero. This gives $\alpha_0 = 0$ and a dependence between
\begin{align*}
- \alpha_1 x^2y^2 + \gamma^2 y^2x^2 \quad \quad \text{and} \quad\quad
x^2y^2 - \alpha_1^2 y^2x^2,
\end{align*}
which implies $\alpha_1^3 = \gamma^2$.

Setting $\alpha_1 = \alpha$ we obtain
\[ d_1^\prime = \left [ \begin{matrix} xy & y^3 \\
- \alpha x^2 & \gamma y^2x \end{matrix} \right ],
d_2^\prime = \left [ \begin{matrix} y^2 \\ \gamma x
                   \end{matrix} \right ],
\]
 where $\gamma^2 = \alpha^3$. Then the relations
(\ref{FibLigEn})
and (\ref{FibLigTo})
are reduced straightforwardly to
\begin{equation}
\label{FibLigRelen}
 x^2y - \alpha yx^2, \, xy^3 + \gamma y^3x.
\end{equation}
If $\alpha$, equivalently $\gamma$, is nonzero then each of the monomials  $x^2$ and $y^3$
commutes with any word up to adjusting with constants.

In the next part we assume that $\alpha$ and $\gamma$ are nonzero.

\medskip
\noindent{Part 3.}
Consider the subcomplex of (\ref{FibLigResAk}) given
by
\begin{equation} \label{FibLigSeks}
\underset{\scriptsize{\begin{matrix} (0,0) \end{matrix}}}
{A} \vmto{d_0}
\underset{\scriptsize{\begin{matrix} (1,0)\\ (0,1)
\end{matrix}}} {A^2} \vmto{d_1^{\prime \prime}}
\underset{\scriptsize{\begin{matrix} (2,1) \\ (1,3) \\ (3,4)
 \end{matrix}}} {A^3} \vmto{d_2^{\prime \prime}}
\underset{\scriptsize{\begin{matrix} (2,3) \\
(4,4) \\
(3,5)
\end{matrix}}}
{A^{3}}
\end{equation}
where
\[ d_1^{\prime \prime} = \left [ \begin{matrix} xy & y^3 & P \\
-\alpha x^2 & \gamma y^2x & -Q \end{matrix} \right ], \,
d_2^{\prime \prime} = \left [
\begin{matrix} y^2 & R_1 & R_2 \\
\gamma x & S_1 & S_2 \\
0 & \mu_1 x & \mu_2 y
\end{matrix} \right ]
\]
We make the following adjustments to $d_1^{\prime \prime}$, noting that
$P$ has bidegree $(2,4)$. By i) subtracting from $P$ right multiplicities
of the first two columns of $d_1^{\prime \prime}$ and ii) using the
relations (\ref{FibLigRelen}) we may assume
\[ P = a_0 y^2xy^2x + a_1 y^2xyxy + a_2 yxy^2xy + a_3 yxyxy^2. \]
Also $R_2$ has bidegree $(1,4)$. By subtracting from $R_2$ right multiplicities
of the first column in $d_2^{\prime \prime}$ and using the relations
(\ref{FibLigRelen}) we may assume $R_2 = 0$.
Furthermore $S_1$ has bidegree $(3,1)$ and again by subtracting right
multiplicities of the first column and using the relations (\ref{FibLigRelen})
we may assume $S_1 = 0$.
The matrices are now
\[ d_1^{\prime \prime} = \left [ \begin{matrix} xy & y^3 & P \\
-\alpha x^2 & \gamma y^2x & -Q \end{matrix} \right ], \,
d_2^{\prime \prime} = \left [
\begin{matrix} y^2 & R_1 & 0 \\
\gamma x & 0 & S_2 \\
0 & \mu_1 x & \mu_2 y
\end{matrix} \right ].
\]
If $\mu_1 = 0$ we get a relation $x^2R_1$ of bidegree $(4,3)$. This must
be a consequence of the relations (\ref{FibLigRelen}) which easily
gives $R_1 = 0$. Similarly if $\mu_2 = 0$ we get a relation
$y^3 S_2 = 0$ of bidegree $(2,5)$ which again easily gives $S_2 = 0$.
Hence both $\mu_1$ and $\mu_2$ must be nonzero, and by a base change
of the generators we may assume they are both $-1$.
Multiplying the matrices above we then get relations
\begin{equation}
\label{relations_eq}
\begin{array}{lc}
xyR_1 = Px, \, & \text{ of bidegree } (3,4) \\
\alpha x^2R_1 = Qx, \, & \text{ of bidegree } (4,3) \\
y^3 S_2 = Py, \, & \text{ of bidegree } (2,5) \\
-\gamma y^2x S_2 = Qy,\, & \text{ of bidegree } (3,4).
\end{array}
\end{equation}
The relations of bidegree $(2,5)$  and $(4,3)$ must be a consequence of the
relations (\ref{FibLigRelen}). The only way this is possible for the
$(2,5)$ relation is if
$a_0 = a_1 = a_2 = 0$ in $P$, since the corresponding terms in $Py$ cannot
be rearranged. Hence we may assume $P = ayxyxy^2$
where $a = a_3$. We then easily see that $S_2 = ayxyx$.

For the relation of bidegree $(4,3)$ to hold, $Q$ must have the form
\[ \alpha x^2 T_0 + \alpha^4 T_1^\prime x,  \]
where $T_0$ and $T_1^\prime$ are bihomogeneous elements of $K\langle x,y\rangle.$

Since $x^2$ commutes up to coefficient change, we may as well assume
that $T_1^\prime$ ends with $y$ and also does not contain $x^2$ as a subword,
so we may write
\[ Q = \alpha x^2 T_0 + \alpha^4 T_1 yx. \]
Considering the relation above of bidegree $(4,3)$, shifting $x^2$ to
the left in $Qx$ and adjusting the coefficients,
we obtain
\[ R_1 = T_0 x + T_1 y, \]
since we may cancel $x^2$.

\medskip
The added defining relation of degree $(3,4)$ of $A$ coming from
the product $d_0.d_1$ is the relation
\begin{equation}
\label{FibLigRelto}  xP - yQ = axyxyxy^2 - yQ.
\end{equation}

The last relation
of bidegree $(3,4)$  in (\ref{relations_eq})  is
\[ \alpha x^2 T_0 y + \alpha^4 T_1 yxy + a \gamma y^2 xyxyx. \]
If nonzero, the term $axyxyxy^2$ of the defining relation (\ref{FibLigRelto}) of
bidegree $(3,4)$ cannot be rearranged using
the first relations (\ref{FibLigRelen}). Moreover,  it does not occur in the
relation above and therefore,
it must be
a consequence of the first defining relations (\ref{FibLigRelen}).
This is also impossible, since
last term cannot cancel. It follows then that $a = 0$ and $P$ and $S_2$ are zero.

\medskip
The three relations of bidegree $(3,4)$ listed
with the defining
relation first are:

\begin{align*}
yx^2 T_0 + \alpha^3yT_1 yx & = 0 \\
x^2T_0 y + \alpha^3T_1 yxy & = 0 \\
xy T_0 x + xy T_1 y & = 0.
\end{align*}
Note that $T_1$ has bidegree $(2,2)$ and does not contain $x^2$ as
subword, so $T_1$ can only contain the terms $m = xyxy, xy^2x$, or
$yxyx$. But then the term $xymy$ in the last relation above cannot
be rearranged by (\ref{FibLigRelen}). Hence it should occur in the
first equation which it does not. Therefore it must be that $T_1 =
0$ and then we easily see that $T_0 = 0$ and so $Q = 0$.
But it is impossible that both
$P$ and $Q$ are zero.

\medskip
\noindent{Part 4.}
Now we consider the case when  (\ref{FibLigFem}) is
identically zero, that is  $\alpha_1 = \gamma = 0$.
Then we obtain (letting $\alpha_0 = -\alpha$)
\[ d_1^{\prime} = \left [ \begin{matrix} xy-\alpha yx & y^3  \\
0 &  -\alpha xy^2  \end{matrix} \right ], \,\,
d_2^{\prime } = \left [
\begin{matrix} y^2 \\
0
\end{matrix} \right ].
\]
We again consider the subcomplex (\ref{FibLigSeks}),
where now
\[ d_1^{\prime \prime} = \left [ \begin{matrix} xy-\alpha yx & y^3 & P \\
0 & -\alpha xy^2 & -Q \end{matrix} \right ], \,\,
d_2^{\prime \prime} = \left [
\begin{matrix} y^2 & R_1 & R_2 \\
0 & S_1 & S_2 \\
0 & \mu_1 x & \mu_2 y
\end{matrix} \right ].
\]
By almost the same type of arguments as in Part 3. we may assume
that $R_2 = 0.$
We work out that $S_2 = 0$ and we may assume
$\mu_2y = -y$. Then it follows quickly that
$P$ is a multiple of $yxyxy^2 - \alpha yxy^2xy$.

Next we show that $Q, S_1$ and $\mu_1$  are all zero and
$R_1 = xyxy^2 - \alpha xy^2xy$. This gives the matrices
\[ d_1^{\prime \prime} = \left [ \begin{matrix} xy-\alpha yx & y^3
& yxyxy^2- \alpha yxy^2xy \\
0 & -\alpha xy^2 &  0 \end{matrix} \right ], \]
\[
d_2^{\prime \prime} = \left [
\begin{matrix} y^2 & xyxy^2- \alpha xy^2xy & 0 \\
0 & 0 & 0 \\
0 & 0 & -y
\end{matrix} \right ]
\]

Part 5. Further computation shows that we must have
\[ d_1 = \left [ \begin{matrix} xy-\alpha yx & y^3
& yxyxy^2- \alpha yxy^2xy & yxy^2xy^2xy^2 - \alpha y^2xyxy^2xy^2\\
0 & -\alpha xy^2 &  0 & 0\end{matrix} \right ].
\]

But now computing the resolution of the algebra with the relations
we get from $d_0 \cdot d_1$, we see that it is not the desired
resolution (\ref{FibLigResAk}). In particular the kernel of $d_1$ has
a syzygy of degree $13$, which is not the case in (\ref{FibLigResAk}).
\end{proof}

{\bf Acknowledgments}.
The first author worked on this paper during her visit to Max Planck
Institute for Mathematics, Bonn, in 2011-2012. It is her pleasant duty to
thank MPIM both for the support and for the inspiring and creative
atmosphere during her visit.


\begin{thebibliography}{1}

\bibitem{Anick82}
{\sc{D.~Anick}}, {\em Noncommutative graded algebras and their
Hilbert series},  J. Algebra {\bf 78} (1982),~120--140.


\bibitem{Anick85}
{\sc{D.~Anick}}, {\em On monomial algebras of finite global
dimension},  Trans. AMS {\bf 291} (1985), ~291--310.


\bibitem{Anick86}
{\sc{D.~Anick}}, {\em On the homology of associative algebras},
  Trans. AMS {\bf 296} (1986),~641--659.


\bibitem{AS}
{\sc{M.~Artin and W.~Schelter}}, {\em Graded algebras of global
dimension 3},  Adv. in Math. {\bf 66} (1987),~171--216.


\bibitem{ATV1}
{\sc{M.~Artin, J.~Tate, and M.~Van~den Bergh}}, {\em Some algebras
associated to automorphisms of elliptic curves}, in The Grothendiek
Festschrift, Vil I, Progr. Math. \textbf{86}, Birkh\"{a}user,
Boston, 1990, ~33--85.

\bibitem{ATV2}
{\sc{M.~Artin, J.~Tate, and M.~Van~den Bergh}}, {\em Modules over
regular algebras of dimension $3$}, Invent. Math. {\bf 106} (1991),
~335--388.



\bibitem{B}
{\sc{G.~ M.~ Bergman}}, {\em The diamond lemma for ring theory},
  Adv. in Math. {\bf 29} (1978), ~178--218.


\bibitem{borisenko85}
{\sc{V.~ V.~ Borisenko}}, {\em Matrix representations of finitely
presented algebras defined by a finite set of words}, Usp. Mat. Nauk
{\bf 35} (1980), ~225--226.


\bibitem{FlVa}
{\sc{G.~ Fl{\o}ystad, J.~E.~ Vatne}}
{\em Artin-Schelter regular algebras of dimension five}, Banach Center Publications, {\bf 93} (2011), ~19--39.

\bibitem{TGI89}
{\sc{T.~ Gateva-Ivanova}}, {\em Global dimension of associative
algebras}, Applied Algebra, Algebraic Algorithms and
Error-Correcting Codes, Lecture Notes in Computer Science,
\textbf{357} (1989), ~213--229,
DOI: 10.1007/3-540-51083-4-61

\bibitem{TGI0512}
{\sc{T.~ Gateva-Ivanova}}, {\em Quadratic algebras, Yang-Baxter
equation, and Artin-Schelter regularity},  Adv. in Mathematics {\bf
230} (2012) ~2152-2175. DOI:10.1016/j.aim.2012.04.016

\bibitem{TGI0612}
{\sc{T.~ Gateva-Ivanova}}, {\em Associative algebras with
obstructions consisting of Lyndon words}, 2012 (Preprint)

\bibitem{LalondeRam}
{\sc{P.~ Lalonde, A.~ Ram}}, {\em Standard Lyndon bases of Lie
algebras and enveloping algebras},
 Trans. AMS {\bf 347} (1995), ~1821--1830.

\bibitem{Lothaire}
{\sc{M.~Lothaire}}, {\em Combinatorics on Words}, Encyclopedia if
Mathematics and its Applications, \textbf{v.17 } GC Rota, Editor
Addison-Wesley Publishing Company, 1983.


\bibitem{Lothaire02}
{\sc{M.~Lothaire}}, {\em Algebraic Combinatorics on Words},
Encyclopedia if Mathematics and its Applications, \textbf{v.90 } GC
Rota, Cambridge University Press, 2002.


\bibitem{Reutenauer}
{\sc{C.~Reutenauer}}, {\em Free Lie algebras},
 London Math. Soc. Monographs (N.S.){\bf 7} Oxford Univ. Press, 1993.



\bibitem{Shirshov1}
{\sc{A.I. ~Shirshov}}, {\em On rings with identity relations},
Mat. Sb. \textbf{43} (1957),   ~277--283.
English transl. in Amer. Math. Soc. Transl., v. 119(1983).

\bibitem{Ufnarovski}
{\sc{V.A. Ufnarovski}}, {\em Combinatorial and asymptotic methods in
algebra. Algebra VI}, ~1--196. Enciclopaedia Math. Sci.  {\bf 57},
Springer, Berlin, 1995

\bibitem{Ufnarovski82}
{\sc{V.A. Ufnarovski}}, {\em Criterion rosta grafov i algebr
zadsannyh slovami}, Mat Zametki, \textbf{31:3}, (1982) ~465--472.
(Russian)

\bibitem{Ufnarovski97}
{\sc{S. Cojocaru, V. Ufnarovski}},  {\em BERGMAN under MS-DOS and Anick's
resolution}, Discrete Mathematics and Theoretical Computer Science \textbf{1}, (1997), ~139--147
\end{thebibliography}
\end{document}